\documentclass[11pt,a4paper,twoside,reqno]{amsart}
\title[]
{Gravitating vortices with positive curvature}

\author[M. Garcia-Fernandez]{Mario Garcia-Fernandez}
\address{Dep. Matem\'aticas\\ Universidad Aut\'onoma de Madrid\\ and
  Instituto de Ciencias Matem\'aticas (CSIC-UAM-UC3M-UCM)\\ Ciudad
  Universitaria de Cantoblanco\\ 28049 Madrid, Spain}
\email{mario.garcia@icmat.es}

\author[V. P. Pingali]{Vamsi Pritham Pingali}
\address{Department of Mathematics, Indian Institute of Science, Bangalore, India - 560012}
\email{vamsipingali@iisc.ac.in}

\author[C.-J. Yao]{Chengjian Yao}
\address{ Institute of Mathematical Sciences, ShanghaiTech University, 393 Middle Huaxia Road, Pudong, 
	Shanghai, 201210 China.}
\email{yaochj@shanghaitech.edu.cn}

\thanks{The work of the first author (Garcia-Fernandez) was partially supported by the Spanish MINECO under grant No. MTM2016-81048-P. The first author is grateful to IMS at ShanghaiTech University for the support provided during a visit in the summer 2019. The second author (Pingali) is partially supported by an SERB MATRICS grant : MTR/2020/000100 and also by grant F.510/25/CAS-II/2018(SAP-I) from UGC (Govt. of India).}

\usepackage{hyperref,url, bm}
\usepackage{amssymb,latexsym}
\usepackage{amsmath,amsthm}
\usepackage[latin1]{inputenc}
\usepackage{enumerate}
\usepackage[all]{xy} \CompileMatrices
\SelectTips{cm}{12} 

\usepackage{verbatim} 

\usepackage{wrapfig}
\usepackage{graphicx}
\usepackage{color}

\def\YYint#1#2#3{{\setbox0=\hbox{$#1{#2#3}{\int}$}
    \vcenter{\hbox{$#2#3$}}\kern-.52\wd0}}

\theoremstyle{plain}
\newtheorem{theorem}{Theorem}[section]
\newtheorem{lemma}[theorem]{Lemma}
\newtheorem{corollary}[theorem]{Corollary}
\newtheorem{proposition}[theorem]{Proposition}

\newtheorem*{theorem*}{Theorem}
\theoremstyle{definition}
\newtheorem{definition}[theorem]{Definition}
\newtheorem{definition-theorem}[theorem]{Definition-Theorem}
\newtheorem{example}[theorem]{Example}

\theoremstyle{remark}
\newtheorem{remark}[theorem]{Remark}

\numberwithin{equation}{section} 

\newcommand{\Id}{\operatorname{Id}}

\newcommand{\Aut}{\operatorname{Aut}}

\newcommand{\dbar}{\bar{\partial}}

\newcommand{\CC}{{\mathbb C}}
\newcommand{\PP}{{\mathbb P}}

\newcommand{\RR}{{\mathbb R}}

\renewcommand{\(}{\left(}
\renewcommand{\)}{\right)}
\newcommand{\vol}{\operatorname{vol}}
\newcommand{\Vol}{\operatorname{Vol}}
\newcommand{\defeq}{\mathrel{\mathop:}=} 

\newcommand{\surj}{\to\kern-1.8ex\to}

\newcommand{\lto}{\longrightarrow}
\newcommand{\lra}[1]{\stackrel{#1}{\longrightarrow}}

\newcommand{\cF}{\mathcal{F}}
\newcommand{\cV}{\mathcal{V}}

\newcommand{\cO}{\mathcal{O}}

\newcommand{\Lie}{\operatorname{Lie}}

\newcommand{\cH}{\mathcal{H}} 
\newcommand{\LieH}{\Lie\cH}

\newcommand{\GL}{\operatorname{GL}}
\newcommand{\SL}{\operatorname{SL}}

\renewcommand{\Im}{\operatorname{Im}}

\begin{document}

\begin{abstract}
We give a complete solution to the existence problem for gravitating vortices with non-negative topological constant $c \geqslant 0$. Our first main result builds on previous results by Yang and establishes the existence of solutions to the Einstein-Bogomol'nyi equations, corresponding to $c=0$, in all admissible K\"ahler classes. Our second main result completely solves the existence problem for $c>0$. Both results are proved by the continuity method and require that a GIT stability condition for an effective divisor on the Riemann sphere is satisfied. For the former, the continuity path starts from a given solution with $c = 0$ and deforms the K\"ahler class. For the latter result we start from the established solution in any fixed admissible K\"ahler class and deform the coupling constant $\alpha$ towards $0$. A salient feature of our argument is a new bound $S_g \geqslant c$ for the curvature of gravitating vortices, which we apply to construct a limiting solution along the path via Cheeger-Gromov theory.
\end{abstract}

\maketitle

\setlength{\parskip}{5pt}
\setlength{\parindent}{0pt}


\section{Introduction}\label{sec:intro}

This work is concerned with the existence of Abelian vortices on a compact Riemann surface $\Sigma$ with back-reaction on the metric. The \emph{vortex equation} 
\begin{equation}\label{eq:vortexeqintro}
i\Lambda_\omega F_h + \frac{1}{2}(|\bm\phi|_h^2-\tau) = 0,
\end{equation}
for a Hermitian metric $h$ on a line bundle $L$ over $\Sigma$ with section $\bm\phi \in H^0(\Sigma,L)$, is a generalization of the equations on $\RR^2$ which were introduced in 1950 by Ginzburg and Landau \cite{GL} in the theory of superconductivity. Abelian vortices have been extensively studied in the mathematics literature after the seminal work of Jaffe and Taubes~\cite{Jaffe-Taubes,Taubes1} on the Euclidean plane, and Witten \cite{Witten} on the 2-dimensional Minkowski spacetime. A complete answer to the existence problem for~\eqref{eq:vortexeqintro} when $\Sigma$ is compact was established independently by Noguchi, Bradlow, and Garc\'ia-Prada ~\cite{Brad,G1,Noguchi}: an Abelian vortex exists if and only if the following inequality is satisfied
\begin{equation}\label{eq:ineqintro}
\frac{4\pi N}{\tau} < \Vol_\omega,
\end{equation}
where $N = \int_\Sigma c_1(L)$, $\omega$ is the volume form of the Riemannian metric $g$ on $\Sigma$, and $\Vol_\omega\defeq\int_\Sigma\omega$.

Following these classical works, a question which has recently emerged is whether a solution of \eqref{eq:vortexeqintro} produces a back-reaction on the 
metric $g$ on $\Sigma$ (with K\"ahler form $\omega$). Returning to the original motivation for the vortex equation in theoretical physics, this question is very natural, as it accounts for a mathematical explanation of \emph{gravitational effects} on the vortex. A concrete proposal for \emph{gravitating vortices}  was put forward by the first author jointly with \'Alvarez-C\'onsul and Garc\'ia-Prada in \cite{AlGaGa2}, in the form of the following coupled equations
\begin{equation}\label{eq:gravvortexeq1intro}
\begin{split}
i\Lambda_\omega F_h + \frac{1}{2}(|\bm\phi|_h^2-\tau) & = 0,\\
S_\omega + \alpha(\Delta_\omega + \tau) (|\bm\phi|_h^2 -\tau) & = c,
\end{split}
\end{equation}
where $\alpha, \tau \in \RR$ are non-negative constants, $\Lambda_{\omega}$ is the trace operator, and the Laplacian is positive-definite by convention. The constant $c$ in the second equation in \eqref{eq:gravvortexeq1intro} is topological, as it is given by the following formula
\begin{equation}\label{eq:constantcintro}
c = \frac{2\pi(\chi(\Sigma) - 2\alpha\tau N)}{\Vol_\omega}.
\end{equation}

The gravitating vortex equations \eqref{eq:gravvortexeq1intro} are fundamental, in the following sense: firstly, they admit a Hamiltonian 
interpretation \cite{AlGaGa,AlGaGaPi} akin to the existence problem for K\"ahler-Einstein metrics, where algebro-geometric stability obstructions appear on compact K\"ahler manifolds with $c_1 > 0$. Secondly, being a particular case of the K\"ahler-Yang-Mills equations \cite{AlGaGa}, the coupled system \eqref{eq:gravvortexeq1intro} is motivated by the question of understanding moduli spaces for smooth polarised varieties equipped with vector bundles. In the present setup, the gravitating vortex equations provide an analytical approach for the moduli space parametrizing Riemann surfaces equipped with an effective divisor. Finally, for $c = 0$ in \eqref{eq:constantcintro}, the system \eqref{eq:gravvortexeq1intro} is equivalent to the Einstein-Bogomol'nyi equations on a Riemann surface (also known as the self-dual Einstein-Maxwell-Higgs equations \cite{HanSohn}) and has a physical interpretation  \cite{AlGaGa2,Yang}. Solutions of this last set of equations are known in the physics literature as Nielsen-Olesen cosmic strings \cite{NielsenOlesen}, and describe a special class of solutions of the Abelian Higgs model coupled with gravity in four dimensions. In this setup, $\tau > 0$ is a \emph{symmetry breaking parameter} in the theory, $\alpha/2\pi$ equals the \emph{gravitational constant}, and $\bm\phi$ represents physically the \emph{Higgs field}.
 
The existence problem for gravitating vortices was first studied by Yang in the case $c = 0$ \cite{Yang,Yang3}. Yang proved a general existence result (Theorem \ref{th:Yang}) which shows that the locations of the zeros of $\bm\phi$ play an important role in global existence. Nonetheless, as the recent result of Han-Sohn \cite{HanSohn} showed, Yang's existence result does not exhaust all the possible solutions (see a detailed discussion in Section \ref{subsec:eb}). For $c \geqslant 0$, the first two authors jointly with \'Alvarez-C\'onsul and Garc\'ia-Prada found a new obstruction to the existence of solutions of \eqref{eq:gravvortexeq1}, thus establishing a relation with Geometric Invariant Theory (GIT) for these equations \cite{AlGaGa2,AlGaGaPi}. For $c <0$, the existence and uniqueness of solutions has been established in \cite{AlGaGaPi} in genus greater than one for a suitable range of the coupling constant $\alpha$, depending only on the topology of the surface and the line bundle. For $c >0$, the analytical techniques in \cite{AlGaGaPi,HanSohn,Yang} do not apply, and the existence problem has hitherto remained open. 

The main goal of the present paper is twofold. The first is to prove the existence of solutions in the case $c=0$ for all admissible K\"ahler classes, that is, satisfying the inequality \eqref{eq:ineqintro}. The second is to provide a complete solution of the existence problem in the case $c > 0$ based on the solution in the case $c=0$. One notable feature is that the system of equations reduces to a single PDE in the case $c=0$, but the case $c>0$ is a truly coupled system of PDE.  

To state our main results, we make a basic observation about the system \eqref{eq:gravvortexeq1intro} which plays an important role in the present work: any solution of \eqref{eq:gravvortexeq1intro} satisfies (see Lemma \ref{lem:curvature-derivative-estimate})
\begin{equation}\label{eq:Slowerbound}
S_g \geqslant c.
\end{equation}
Furthermore, by \eqref{eq:constantcintro}, the condition $c \geqslant 0$ implies that $\Sigma \cong \PP^1,
$
and, in this setup, it is equivalent to the following constraint in the coupling constant
\begin{equation}\label{eq:maxinterval}
\alpha \in [0,\tfrac{1}{\tau N}].
\end{equation}
Our first main result is about the existence of solutions to the Einstein-Bogomol'nyi equations, which correspond to the value $\alpha = \tfrac{1}{\tau N}$ (equivalently, $c = 0$).
\begin{theorem}
	\label{intro:thm{existence}}
	Let $D=\sum_j n_jp_j$ be the effective divisor on $\PP^1$ corresponding to $(L,\bm\phi)$.  Suppose $D$ is GIT polystable for the 
	$\SL(2,\CC)$-action on the space of effective divisors. Then, for any $V> \frac{4\pi N}{\tau}$, there exists a solution $(\omega, h)$ to the Einstein-Bogomol'nyi equations such that $\text{Vol}_{\omega}=V$. 
\end{theorem}

Our second main result establishes the existence of solutions for the system \eqref{eq:gravvortexeq1intro} when $c>0$.

\begin{theorem}\label{th:mainintro}
Suppose $\alpha \in (0,\tfrac{1}{\tau N})$. Let $D=\sum_j n_jp_j$ be the effective divisor on $\PP^1$ corresponding to $(L,\bm\phi)$. Suppose that $D$ is GIT polystable for the 
$\SL(2,\CC)$-action on the space of effective divisors. Then, for any $V> \frac{4\pi N}{\tau}$, there exists a solution $(\omega, h)$ to the gravitating vortex equations with coupling constant $\alpha$ such that $\text{Vol}_{\omega}=V$. 
\end{theorem}

Our main results give a converse of \cite[Theorem 1.3]{AlGaGaPi}, by the first two authors jointly with \'Alvarez-C\'onsul and Garc\'ia-Prada, which established the GIT polystability of the divisor $D$ assuming the existence of \eqref{eq:gravvortexeq1intro}. Thus, combined with this result, 
Theorem \ref{intro:thm{existence}} and Theorem \ref{th:mainintro} provide a complete solution to the existence problem for gravitating vortices with $c \geqslant 0$.

When $D = \tfrac{N}{2} \cdot p_1 + \tfrac{N}{2} \cdot p_2$ (i.e. $D$ is strictly polystable), the statement of Theorem \ref{intro:thm{existence}} follows from the study of the volume of Yang's symmetric solution \cite{Yang3} (see Proposition \ref{prop:YangODEvolumeestimate}). For $D$ stable, Theorem \ref{intro:thm{existence}} is proved via a continuity method 
starting from any solution of the Einstein-Bogomol'nyi equations  constructed by Yang in \cite{Yang} and deforming the volume.  Our method of proof of Theorem \ref{th:mainintro}  also exploits a continuity path starting from a solution with $\alpha = \tfrac{1}{\tau N}$, and deforming the coupling constant $\alpha$ towards $0$. Since these two continuity methods are similar, we provide a detailed proof of Theorem \ref{th:mainintro} in Section \ref{sec:open} and Section \ref{sec:close} and specify the key changes needed for the proof of Theorem \ref{intro:thm{existence}} in Section \ref{sec:eb}. To prove that the existence of solutions of \eqref{eq:gravvortexeq1intro} is an open condition along the path, we distinguish two cases. The case when the support of the divisor $D$ has more than two points follows easily by application of the implicit function theorem (see Lemma \ref{lem:openst}). Our proof breaks down in the strictly polystable case 
due to the presence of symmetries. Motivated by this, in Definition \ref{def:extremal_pair} we introduce a notion of \emph{extremal pair}, which provides an analogue for the gravitating vortex equations of the familiar notion of an extremal metric in K\"ahler geometry. With this definition at hand, the proof of the strictly polystable case follows by a Lebrun-Simanca type argument \cite{AlGaGa2,LS1} (see Lemma \ref{lem:openpolyst}), combining the $\alpha$-Futaki invariant, introduced in \cite{AlGaGaPi}, with a Matsushima-Lichnerowicz type theorem for the gravitating vortex equations (see \cite[Theorem 3.6]{AlGaGa3}).

As for closedness, the $C^0$ estimate for the K\"ahler potentials along the path is obstructed by 
the GIT polystability of the divisor $D$ \cite[Theorem 1.3]{AlGaGaPi}. To tackle this problem, in Section \ref{subsec:RGV} we relate \eqref{eq:gravvortexeq1intro} to a different set of equations that we call the \emph{Riemannian gravitating vortex equations} (see Definition \ref{def:RGV}). These new equations get rid of the dependence on the line bundle at the cost of introducing singularities. Then, we derive an \emph{a priori} $C^1$ estimate for the scalar curvature of a solution of \eqref{eq:RGV} in Section \ref{subsec:apriori}, using which we obtain a Cheeger-Gromov limit. The diffeomorphisms involved in taking the limit are not necessarily holomorphic, but we use a slice theorem and the uniqueness of almost complex structure on $S^2$ to promote them to a sequence of holomorphic automorphisms of $\mathbb{P}^1$ (see Lemma \ref{lem:CGamm}). A delicate analysis of the Green's functions along the sequence allows us to show that the amended Cheeger-Gromov limit is a solution of the Riemannian gravitating vortex equations. 
Thus, the limit divisor is polystable by \cite{AlGaGaPi} and must then be inside the $\SL(2,\mathbb{C})$-orbit of $D$.


It is interesting to observe that the estimates in the proof of our main result work even as we approach $\alpha \to 0$, producing a solution of the gravitating vortex equations on $\mathbb{P}^1$ with $\alpha = 0$. Since the system \eqref{eq:gravvortexeq1intro} decouples in this limit, \cite[Theorem 1.3]{AlGaGaPi} does not apply and we are not able to conclude that the limiting divisor lies in the $\SL(2,\CC)$-orbit of $D$. The striking difference between the existence of gravitating vortices and the existence of (simply) vortices is that the latter does not impose any stability condition on the divisor. In the other extreme of the interval, when $\alpha \to \tfrac{1}{\tau N}$ (and hence $c \to 0$), our estimates collapse and we are not able to provide new information about Yang's solutions \cite{Yang,Yang3}.



We expect that the methods introduced in the present paper can be used to prove compactness of the moduli space of gravitating vortices with $c \geqslant 0$ (with moving complex structure and divisor, and fixed coupling constant $\alpha$) and the existence of a continuous surjective map from the moduli space onto the space of binary quantics $S^N \PP^1 /\!\!/ \operatorname{SL}(2,\CC)$.  We speculate that the study of the families of moduli spaces as $\alpha \to \tfrac{1}{\tau N}$ may yield a method for understanding the more difficult moduli space of solutions of the Einstein-Bogomol'nyi equations (where $c = 0$), which plays a key role in the physical theory of cosmic strings \cite{Yang,Yang3}. We leave these interesting perspectives for future investigations.

\emph{Acknowledgements}: The authors wish to thank L. \'Alvarez-C\'onsul, V. Datar, Y. Imagi, and O. Garc\'ia-Prada for useful discussions. 

\section{The gravitating vortex equations with $c \geqslant 0$}

In this section we recall the definition of the gravitating vortex equations introduced in \cite{AlGaGa2}, state our main theorems, provide a Riemannian characterization of the equations, and establish a regularity result.

\subsection{Gravitating vortices}

Let $\Sigma$ be a compact connected Riemann surface of arbitrary
genus, $L$ a holomorphic line bundle over $\Sigma$, and $\bm\phi$ a global holomorphic section of $L$. We will assume that $\bm\phi$ is not identically zero, and hence
\begin{equation}\label{eq:N}
N = \int_\Sigma c_1(L) > 0. 
\end{equation}
Fix real constants $\tau > 0$ and $\alpha \geqslant0$, called the \emph{symmetry breaking parameter} and the \emph{coupling constant}, respectively.

\begin{definition}\label{def:gravvorteq}
The gravitating vortex equations, for a K\"ahler metric $g$ on
$\Sigma$ with K\"ahler form $\omega$ and a Hermitian metric $h$ on
$L$, are
\begin{equation}\label{eq:gravvortexeq1}
\begin{split}
i\Lambda_\omega F_h + \frac{1}{2}(|\bm\phi|_h^2-\tau) & = 0,\\
S_\omega + \alpha(\Delta_\omega + \tau) (|\bm\phi|_h^2 -\tau) & = c.
\end{split}
\end{equation}
Solutions of these equations will be called gravitating vortices.
\end{definition}

In~\eqref{eq:gravvortexeq1}, $F_h$ is the curvature 2-form of the
Chern connection of $h$, $\Lambda_\omega F_h\in C^\infty(\Sigma)$ is
its contraction with $\omega$, $|\bm\phi|_h\in C^\infty(\Sigma)$ is the
pointwise norm of $\bm\phi$ with respect to $h$, $S_\omega$ is the scalar
curvature of $\omega$ (as usual, K\"ahler metrics will be identified
with their associated K\"ahler forms), and $\Delta_\omega$ is the
Laplace operator for the metric, given by
\[
\Delta_\omega f = 2i \Lambda_\omega\dbar\partial f,
\]
for $f\in C^\infty(\Sigma)$. Notice that in this convention, $\Delta_\omega$ equals the Hodge Laplacian $\Delta_g$ of the Riemannian metric $g$. 


 The constant $c \in \RR$ is topological,
and is explicitly given by
\begin{equation}\label{eq:constantc}
c = \frac{2\pi(\chi(\Sigma) - 2\alpha\tau N)}{\Vol_\omega},
\end{equation}
with $\Vol_\omega\defeq\int_\Sigma\omega$, as can be deduced
by integrating~\eqref{eq:gravvortexeq1} over $\Sigma$.
 
Given a fixed K\"ahler metric $\omega$, the first equation in \eqref{eq:gravvortexeq1}, that is,
\begin{equation}\label{eq:vortexeq}
i\Lambda_\omega F_h + \frac{1}{2}(|\bm\phi|_h^2-\tau) = 0,
\end{equation}
is the vortex equation for a Hermitian metric $h$ on $L$. 
The existence of solutions of~\eqref{eq:vortexeq}, often called vortices, was established independently by Noguchi, Bradlow, and Garc\'ia-Prada.

\begin{theorem}[{\cite{Brad,G1,G3,Noguchi}}]
\label{th:B-GP}
For every fixed K\"ahler form $\omega$ there exists a solution $h$ of the vortex
equation~\eqref{eq:vortexeq} if and only if
\begin{equation}\label{eq:ineq}
\frac{4\pi N}{\tau} < \Vol_\omega,
\end{equation}
in which case the solution is unique.
\end{theorem}

When $\alpha >0$, finding a solution of the vortex equation \eqref{eq:vortexeq} is not enough to solve the equations in Definition \ref{def:gravvorteq}. As mentioned in Section \ref{sec:intro}, the existence problem for gravitating vortices has been studied and partially solved in \cite{AlGaGa2,AlGaGaPi,Yang,Yang3} when $c \leqslant 0$. The main goal of the present paper is to provide a complete solution of the existence problem in the case $c \geqslant 0$. 

To finish this section we recall an important obstruction to the existence of gravitating vortices with $c \geqslant 0$. This obstruction was found in \cite{AlGaGaPi} and uses Geometric Invariant Theory. Observe first that the existence of gravitating vortices for $c \geqslant 0$ forces the topology of the surface to be that of the $2$-sphere, because  $N > 0$ implies $\chi(\Sigma) > 0$ by~\eqref{eq:constantc}. Thus, up to biholomorphism, we can assume $\Sigma$ to be the Riemann sphere $\PP^1$. 
Notice that, in this setup, the condition $c \geqslant 0$ is equivalent to the following constraint in the coupling constant
\begin{equation}\label{eq:maxinterval}
\alpha \in [0,\tfrac{1}{\tau N}].
\end{equation}
Consider the effective divisor $D = \sum_j n_j p_j$ determined by the pair $(L, \bm\phi)$. Recall that the space of effective divisors of degree $N$ on $\PP^1$ admits a canonical linearised $\SL(2,\CC)$-action.


\begin{theorem}[\cite{AlGaGaPi}]\label{th:AGGP}

If $(\PP^1,L,\bm\phi)$ admits a solution of the gravitating vortex equations with coupling constant $\alpha > 0$, then \eqref{eq:ineq} holds and the divisor $D$ is GIT polystable for the $\SL(2,\CC)$-action.
\end{theorem}

Being ``GIT polystable'' means that $D$ is either stable or strictly polystable. These conditions can be written more explicitly in terms of the multiplicities $n_j$ of the divisor $D$ using the Hilbert-Mumford criterion. 

\begin{proposition}[{\cite[Ch.~4, Proposition~4.1]{MFK}}]\label{prop:GIT}
Consider the space of effective divisors 
on $\PP^1$ with its canonical linearised $\SL(2,\CC)$-action. Let
$D=\sum_j n_jp_j$ be 
an effective divisor, for finitely many different points $p_j\in\PP^1$
and integers $n_j>0$ such that $N=\sum_j n_j$. Then
\begin{enumerate}
\item[\textup{(1)}] $D$ is stable if and only if $n_j < \frac{N}{2}$ for all $j$.
\item[\textup{(2)}] $D$ is strictly polystable if and only if $D=\frac{N}{2}p_1 + \frac{N}{2}p_2$, where $p_1 \neq p_2$ and $N$ is even.
\item[\textup{(3)}] $D$ is unstable if and only if there exists $p_j
  \in D$ such that $n_j>\frac{N}{2}$.
\end{enumerate}
\end{proposition}


\subsection{Main results}
\label{subsec:eb}

In order to present our results, we state first the main existence results for the Einstein-Bogomol'nyi equations in the seminal papers \cite{Yang,Yang3}, and more recently in \cite{HanSohn}. We will do this in a way that is useful for the present paper. As mentioned in Section \ref{sec:intro}, the Einstein-Bogomol'nyi equations correspond to the gravitating vortex equations \eqref{eq:gravvortexeq1} in the special case $c = 0$. Equivalently, a solution of the Einstein-Bogomol'nyi equations is a gravitating vortex on $\mathbb{P}^1$ with $\alpha = \tfrac{1}{\tau N}$ (see \eqref{eq:maxinterval}). 

\begin{theorem}[\cite{HanSohn,Yang,Yang3}]\label{thm:HSY}
Let $D=\sum_j n_jp_j$ be the effective divisor on $\PP^1$ corresponding to a pair $(L,\bm\phi)$.

\begin{enumerate}
\item Assume that $D$ is strictly polystable. Then, for any $V> \frac{4\pi N}{\tau}$, there exists a solution $(\omega, h)$ to the Einstein-Bogomol'nyi equations such that $\text{Vol}_{\omega}=V$. In this case the solution admits a $T^2$-symmetry.

\item Assume that $D$ is stable. Then, for any $V > \frac{4\pi N}{\tau}$ there exists a solution $(\omega, h)$ to the Einstein-Bogomol'nyi equations satisfying $\text{Vol}_{\omega} > V$.

\end{enumerate}
\end{theorem}

The existence of solutions to the Einstein-Bogomol'nyi equations in the present setup was proved first by Yang in \cite{Yang,Yang3}. A detailed study of Yang's solutions in the case that $D$ is stable was undertaken by Han-Sohn \cite{HanSohn}, who (implicitly) proved the behavior of the volume stated in part (2) of Theorem \ref{thm:HSY} (see Lemma \ref{lem:volume-behavior}). Part (1) of Theorem \ref{thm:HSY} follows from Yang's main result in \cite{Yang3} combined with Proposition \ref{prop:YangODEvolumeestimate} below. The $T^2$-symmetry of the solution in $(1)$ will be made more explicit in Section \ref{subsec:openS1}. Further details about the proof of Theorem \ref{thm:HSY} will be given in Section \ref{sec:eb}.

The lower bound $\Vol_\omega > \frac{4\pi N}{\tau}$ on the total volume in Theorem \ref{thm:HSY} is a necessary condition for a K\"ahler class $[\omega] \in H^2(\mathbb{P}^1,\RR)$ to admit a solution to the gravitating vortex equations (and hence to the Einstein-Bogomol'nyi equations) (see Theorem \ref{th:B-GP}). Such K\"ahler classes will be called \emph{admissible}. A subtle point in the statement of Theorem \ref{thm:HSY} is that, when $D$ is stable, it does not provide any information about existence for arbitrary large K\"ahler classes. The first main result of the current work fills this gap, setting the existence of solutions for the Einstein-Bogomol'nyi equations in any admissible K\"ahler class.

\begin{theorem}\label{thm{existence}}
Let $D=\sum_j n_jp_j$ be the effective divisor on $\PP^1$ corresponding to $(L,\bm\phi)$. Suppose that $D$ is GIT stable for the $\SL(2,\CC)$-action on the space of effective divisors. Then, for any $V> \frac{4\pi N}{\tau}$, there exists a solution $(\omega, h)$ to the Einstein-Bogomol'nyi equations such that $\text{Vol}_{\omega}=V$. 
\end{theorem}


This result is achieved by a continuity method, starting with a solution to the Einstein-Bogomol'nyi equations given by part $(2)$ of Theorem \ref{thm:HSY}. Combined with part $(1)$ of Theorem \ref{thm:HSY}, we obtain Theorem \ref{intro:thm{existence}} as stated in the introduction. Notice that the volume of the solutions to the Einstein-Bogomol'nyi equations is related to the concept of \emph{effective radius} in the physics literature $R_{\text{eff}} =\sqrt{\text{Vol}_{\omega}/4\pi}$ (see \cite{Yang2}).  
Thus, in particular Theorem \ref{intro:thm{existence}} solves the problem of determining all possible effective radius for solutions to the  Einstein-Bogomol'nyi equations.

Our second main result, which we will state next, establishes the existence of solutions for the system \eqref{eq:gravvortexeq1intro} when $c>0$ (cf. Theorem \ref{th:mainintro}). In particular, it provides an extension of Theorem \ref{thm{existence}} for gravitating vortices on the Riemann sphere and also a converse for Theorem \ref{th:AGGP}.

\begin{theorem}\label{th:main}
Suppose that $\alpha \in (0,\tfrac{1}{\tau N})$. Let $D=\sum_j n_jp_j$ be the effective divisor on $\PP^1$ corresponding to $(L,\bm\phi)$. Suppose that $D$ is GIT polystable for the $\SL(2,\CC)$-action on the space of effective divisors. Then, for any $V> \frac{4\pi N}{\tau}$, there exists a solution $(\omega, h)$ to the gravitating vortex equations with coupling constant $\alpha$ such that $\text{Vol}_{\omega}=V$. 
\end{theorem}

Our proof of Theorem \ref{th:main} is also via the continuity method. Starting with a solution at $\alpha = \tfrac{1}{\tau N}$ given by Theorem \ref{thm{existence}}, our continuity path deforms the coupling constant in the interval $(0, \frac{1}{\tau N}]$. Since the strategies to implement the continuity method in Theorem \ref{thm{existence}} and Theorem \ref{th:main} are similar, in Section \ref{sec:open} and Section \ref{sec:close} we carry out the proof of Theorem \ref{th:main} assuming Theorem \ref{thm{existence}}. In Section \ref{sec:eb}, we explain the crucial differences in the two situations, detail the necessary modifications, and complete the proof of Theorem \ref{thm{existence}}.

\begin{remark}
In the \emph{weak coupling limit} $\alpha \to 0$ the equations \eqref{eq:gravvortexeq1} decouple and, by Theorem \ref{th:B-GP}, the existence of solutions reduces to the numerical condition \eqref{eq:ineq}. Potentially, one could use the solution with $\alpha = 0$ to start the continuity method and yield an alternative proof of Theorem \ref{th:main}. Nonetheless, in this limit the automorphism group of the solution `jumps' and our argument does not work.
\end{remark}

\subsection{Regularity and the Riemannian viewpoint}\label{subsec:RGV}

In order to address the proof of Theorem \ref{th:main}, it will be useful to look at equations \eqref{eq:gravvortexeq1} from the point of view of Riemannian geometry. This will be used for the construction of a weak solution of the equations along the continuity path using the Cheeger-Gromov convergence theory.

Consider the $2$-sphere $S^2$ and fix an unordered tuple of points with multiplicities
$$
\sum_j n_j p_j \in S^{N}(S^2),
$$
and total degree $ \sum_j n_j  = N >0$.

\begin{definition}\label{def:RGV}
A solution of the Riemannian gravitating vortex equations on $(S^2,\sum_j n_j p_j)$ with coupling constant $\alpha > 0$ is given by a triple $(g,\eta,\Phi)$ such that
\begin{enumerate}

\item $g$ is a smooth Riemannian metric on $S^2$,

\item $\eta$ is a smooth closed real $2$-form on $S^2$ such that $\int_{S^2}\eta  = 2\pi N$, 

\item $\Phi \in C^\infty(S^2)$ is a non-negative function $\Phi \geqslant 0$, called \emph{the state function}, vanishing precisely at the $p_j \in S^2$, and such that $\log \Phi \in L^1_{loc}(S^2)$,

\end{enumerate}

which satisfy the following system of equations

\begin{equation}\label{eq:RGV}
\begin{split}
\eta + \frac{1}{2}(\Phi-\tau) \vol_g & = 0,\\
S_g + \alpha(\Delta_g + \tau) (\Phi -\tau) & = c,\\
\Delta_g \log \Phi & = (\tau-\Phi) - 4\pi\sum_j n_j \delta_{p_j}. 
\end{split}
\end{equation}

\end{definition}

Here, $\delta_{p_j}$ denotes the Dirac delta function at the point $p_j$. The last equation has to be understood in the distributional sense. Our interest in the system \eqref{eq:RGV} is provided by the following basic result.

\begin{lemma}\label{lem:KR}
Any solution $(g,\eta,\Phi)$ of the Riemannian gravitating vortex equations \eqref{eq:RGV} on $(S^2,\sum_j n_j p_j)$ with coupling constant $\alpha$, determines a tuple $(\Sigma,L, \bm\phi,\omega,h)$ (unique up to rescaling of $\bm\phi$ and $h$), where $\Sigma = (S^2,J)$ is the Riemann surface of genus zero determined by the conformal class of $g$, $L$ is a holomorphic line bundle with $\int_\Sigma c_1(L) = N$, $\bm\phi \in H^0(\Sigma,L)$ with fixed divisor $D = \sum_j n_j p_j$, and $(\omega,h)$ is a solution of the gravitating vortex equations with coupling constant $\alpha$ such that $|\bm\phi|_h^2 = \Phi$ and $\eta = i F_h$. Conversely, any such tuple determines a solution $(g,\eta,\Phi)$ of \eqref{eq:RGV} on $(S^2,\sum_j n_j p_j)$ with coupling constant $\alpha$.
\end{lemma}

\begin{proof} Given $(\Sigma,L,\bm\phi,\omega,h)$ as in the statement, we notice that the smooth function $|\bm\phi|_h^2$ satisfies the Poinc\'are-Lelong Formula
\[
i\partial\bar\partial \log |\bm\phi|_h^2 = - iF_h + 2\pi \sum_j n_j [p_j],
\]
where $[p_j]$ denotes the current of integration over the divisor $p_j \in \Sigma$. Thus, $\Phi=|\bm\phi|_h^2$ satisfies \
\[
\Delta_g \log \Phi = 2\Lambda_\omega \left( i\bar\partial \partial \log \Phi \right)
=
(\tau-\Phi) - 4\pi\sum_j n_j \delta_{p_j}.
\]
However, $(\Sigma, L, \bm\phi, \omega, h)$ and $(\Sigma, L, \varepsilon\bm\phi, \omega, |\varepsilon|^{-2}h)$ (with $\varepsilon\in \CC^*$) correspond to the same $(g, \eta, \Phi)$. Taking a holomorphic coordinate $z$ around $p_j$ we have that $\log \Phi$ differs from $\log |z|^{2n_j}$ by a smooth function, and therefore $\log \Phi \in L_{loc}^1(S^2)$.

For the converse, we construct the tuple $(\Sigma, L, \bm\phi, \omega, h)$ from the data $(g, \eta, \Phi)$. The orientation together with $g$ on $S^2$ gives a complex structure $J$, making $S^2$ into a Riemann surface. Take $L$ to be the line bundle defined by the effective divisor $\sum_j n_jp_j$, and $\bm\phi$ to be a defining section of this divisor (there is an ambiguity of global multiple of a nonzero phase). By the first equation in \eqref{eq:RGV} the form $\eta$ is of type $(1,1)$, and furthermore it has integral $2\pi N = 2\pi c_1(L)$, by assumption. Therefore, $- i\eta$ can be realized as the curvature form of some Hermitian metric $h$ on $L$. Since $\log\Phi$ and $\log |\bm\phi|_h^2$ satisfies the same distributional equation, $\log\Phi-\log |\bm\phi|_h^2$ is constant on $S^2$, i.e. $\Phi=\gamma|\bm\phi|_h^2$. By rescaling $h$ by a factor of $\gamma$, we see that $\Phi=|\bm\phi|_h^2$. If we choose $\varepsilon\bm\phi$ instead of $\bm\phi$ as the defining section of $L$, for $\varepsilon\in \CC^*$, the condition $\Phi=|\bm\phi|_h^2$ fixes the Hermitian metric to be $|\varepsilon|^{-2}h$. 

\end{proof}

To finish this section, we prove a regularity result for solutions of the Riemannian gravitating vortex equations, used in Section \ref{sec:close}. Let us spell out precisely the notion of weak of solution of \eqref{eq:RGV} that we will use. With our application in mind, we assume that the conformal class of the metric solution induces the standard almost complex structure $J_0$ in $\mathbb{P}^1$. Given a K\"ahler metric of class $C^{1,\beta}$ on $\mathbb{P}^1$ with volume $\nu =\Vol(S^2, g) = \int_{S^2}\omega$, by fixing a smooth K\"ahler metric $g_0$ with the same volume we can write
$$
g = e^\varphi g_0
$$
for a $C^{1,\beta}$-function on $\PP^1$. Using the identities
\begin{equation}\label{eq:conformal}
e^{\varphi}S_{g} = S_{g_0} + \frac{1}{2}\Delta_{g_0} \varphi, \qquad e^\varphi \Delta_{g} = \Delta_{g_0},
\end{equation}
we can interpret the system \eqref{eq:RGV} in the distributional sense as follows
\begin{equation}\label{eq:RGVweak}
\begin{split}
\eta + \frac{1}{2}(\Phi-\tau) e^\varphi \vol_{g_0} & = 0,\\
S_{g_0} + \frac{1}{2}\Delta_{g_0} \varphi + \alpha (\Delta_{g_0} + \tau e^\varphi) (\Phi -\tau) & = c e^\varphi,\\
\Delta_{g_0} \log \Phi & = e^\varphi (\tau-\Phi) - 4\pi\sum_j n_j e^{\varphi(p_j)}\delta_{p_j} 
\end{split}
\end{equation}
for $(\varphi,\eta,\Phi)$ of class $C^{1,\beta}$. Our notion of weak solution is precisely in this sense.

\begin{lemma}\label{lem:regular2}
Assume that $(g,\eta,\Phi)$ is a weak solution of \eqref{eq:RGV} of class $C^{1,\beta}$, for $0 < \beta < 1$, on $(S^2,\sum_j n_j p_j)$ with coupling constant $\alpha$, such that the Riemann surface of genus zero determined by the conformal class of $g$ is the Riemann sphere $\mathbb{P}^1 = (S^2,J_0)$. Then, $(g,\eta,\Phi)$ is a smooth solution in the sense of Definition \ref{def:RGV}.
\end{lemma}

\begin{proof}
Denote $\nu =\Vol(S^2, g) = \int_{S^2}\omega$, where $\omega = \vol_{g} = e^\varphi \vol_{g_0}$ is the K\"ahler form of $g$. For a choice of $p_j$, consider the Green's function $G_j$ solving the distributional equation
$$
dd^c  G_j  = [p_j]- \frac{1}{\nu} \omega
$$
with the normalization $\int_{S^2} G_j \omega =0$ (see section \ref{Green-function}). Since $\omega$ is of class $C^{1,\beta}$ on isothermal coordinates on $\mathbb{P}^1$, it follows that $G_j$ is of class $C^{3,\beta}$ away from $p_j$. Consider now
$$
G = \sum_{j} n_j  G_j
$$
solving the distributional equation
\begin{equation}\label{multiple-Green-function-0}
\Delta_{g_0} G = \frac{N}{\nu}e^\varphi - \sum_j n_j e^{\varphi(p_j)}\delta_{p_j}.
\end{equation}
Combining this formula with the last equation in \eqref{eq:RGVweak}, it follows that $v =\log \Phi - 4\pi G$ is a solution of the distributional equation
\begin{equation}\label{error-Laplacian-equation-0}
\Delta_{g_0}  v = (\tau- \Phi)e^\varphi  - \frac{4\pi N}{\nu}e^\varphi.
\end{equation}
By assumption, the right hand side is of class $C^{1,\beta}$, and therefore by the Schauder estimates it follows that $v  \in C^{3,\beta}(S^2)$. By construction of $G$, choosing a holomorphic coordinate $z$ on $\mathbb{P}^1$ centered at $p_j$, it follows that
$$
G - \frac{1}{4\pi} \log |z|^{2n_j} \in C^{3,\beta}
$$
locally around this point, and therefore $\Phi =e^{v + 4\pi G} \in C^{3,\beta}(S^2)$. Considering now the second equation in \eqref{eq:RGVweak}, it follows from the Schauder estimates that $\varphi$ (and hence $g$) is of class $C^{3,\beta}$ and, by the first equation in \eqref{eq:RGV}, $\eta$ is also of class $C^{3,\beta}$. The proof now follows by induction, iterating the previous argument.
\end{proof}

\section{The method of continuity and openness}\label{sec:open}

In this section we introduce the continuity path that we will use for the proof of Theorem \ref{th:main}, and prove that the existence of solutions of the gravitating vortex equations is an open condition for the coupling constant and the volume. 

\subsection{Continuity method}\label{subsec:continuity}

Let $L$ be a holomorphic line bundle over the Riemann sphere $\PP^1$ with degree $N$ (see \eqref{eq:N}), and $\bm\phi$ a global holomorphic section of $L$. We assume that $\bm\phi$ is not identically zero, and denote by
$$
D = \sum_j n_j p_j \in S^N(\PP^1)
$$
the corresponding divisor. We consider the gravitating vortex equations \eqref{eq:gravvortexeq1} on $(\PP^1,L,\bm\phi)$ with coupling constant $\alpha \in (0,\tfrac{1}{\tau N}]$, symmetry breaking parameter $\tau > 0$, and unknowns given by pairs $(\omega,h)$ with $\Vol_\omega > \tfrac{4\pi N}{\tau }$. As in \eqref{eq:constantc}, for $V > 0$ we define
\begin{equation*}\label{constants}
c_{\alpha}(V) = \frac{4\pi(1-\alpha\tau N)}{V} \geqslant0.
\end{equation*}

For the proof of Theorem \ref{th:main} we will use the method of continuity, where the continuity path is simply equations~\eqref{eq:gravvortexeq1}, with the coupling constant $\alpha \in (0,\tfrac{1}{\tau N}]$ as the continuity parameter, i.e.,
\begin{equation}\label{eq:continuitypath}
\begin{split}
i\Lambda_\omega F_h + \frac{1}{2}(|\bm\phi|_h^2-\tau) & = 0,\\
S_\omega + \alpha(\Delta_\omega + \tau) (|\bm\phi|_h^2 -\tau) & = c_\alpha(\Vol_\omega).
\end{split}
\end{equation}





For the proof of the following lemma, we do not assume Theorem \ref{thm{existence}}.

\begin{lemma}\label{lem:openst}
Assume that $D$ is GIT stable. Then, the following set is non-empty and open 
\begin{equation*}
S = \{(\alpha,V) \in (0,\tfrac{1}{\tau N}] \times (\tfrac{4 \pi N}{\tau},+\infty ) \; \textrm{such that \eqref{eq:continuitypath} has a smooth solution }(\omega, h) \text{ with }\Vol_\omega = V \}.
\end{equation*}
\end{lemma}

\begin{proof}
Let $\alpha, \varepsilon > 0$ and fix a pair $(\omega_0,h_0)$. Consider the operator
\begin{equation}
\label{eq:Talphaoperator}
\mathbf{T}_{\alpha, \varepsilon}= (\mathbf{T}_{\alpha, \varepsilon}^0,\mathbf{T}_{\alpha, \varepsilon}^1)
\colon C^\infty(\PP^1) \times C^\infty(\PP^1) \lto C^\infty(\PP^1) \times C^\infty(\PP^1),
\end{equation}
given by 
\begin{align*}
\mathbf{T}_{\alpha, \varepsilon}^0(u,f) & = i\Lambda_\omega F_h + \frac{1}{2}|\bm\phi|^2_h - \frac{\tau}{2},\\
\mathbf{T}_{\alpha, \varepsilon}^1(u,f) & = - S_\omega - \alpha \Delta_\omega |\bm\phi|^2_h + 2\alpha\tau i\Lambda_\omega F_h,
\end{align*}
where $(\omega,h) = (\varepsilon(\omega_0 + dd^c u),e^{2f}h_0)$. By the proof of \cite[Lemma 6.3]{AlGaGaPi}, the linearization of $\mathbf{T}_{\alpha, \varepsilon}$ at $(u,f)$ satisfies
\begin{align*}
\delta \mathbf{T}_{\alpha, \varepsilon}^0(\dot u,\dot f) & = \mathbf{L}^0_{\alpha, \varepsilon}( \varepsilon \dot u,\dot f) + \varepsilon J\eta_{\dot{u}}\lrcorner d(\mathbf{T}_{\alpha, \varepsilon}^0(u,f)),\\
\delta \mathbf{T}_{\alpha, \varepsilon}^1(\dot u,\dot f) & = \mathbf{L}^1_{\alpha, \varepsilon}(\varepsilon \dot u,\dot f) + (d(\mathbf{T}_{\alpha, \varepsilon}^1(u,f)),\varepsilon d\dot u)_\omega,
\end{align*}
where $\mathbf{L}_{\alpha, \varepsilon} = \mathbf{L}_{\alpha, \varepsilon,u,f}= (\mathbf{L}_{\alpha, \varepsilon}^0,\mathbf{L}_{\alpha, \varepsilon}^1)$ is the linear differential operator
\begin{equation}
\label{eq:Lalphaoperator}
\mathbf{L}_{\alpha, \varepsilon} \colon C^\infty(\PP^1) \times C^\infty(\PP^1) \lto C^\infty(\PP^1) \times C^\infty(\PP^1),
\end{equation}
defined by
\begin{equation}
\label{eq:Dalphapaoperator}
\begin{split}
\mathbf{L}_{\alpha, \varepsilon}^0(\dot u,\dot f) & = d^*(d \dot f + \eta_{\dot u}\lrcorner i F_h) + (\bm\phi, -J \eta_{\dot u}\lrcorner d_A \bm\phi + \dot f \bm\phi)_h,\\
\mathbf{L}_{\alpha, \varepsilon}^1(\dot u,\dot f) & = \operatorname{P}^*\operatorname{P} \dot u - 4\alpha i \Lambda_\omega d (d_A \bm\phi , -J \eta_{\dot u}d_A \bm\phi + \dot f \bm\phi)_h \\
& - 2 \alpha i \Lambda_\omega d ((d \dot f + \eta_{\dot u}\lrcorner i F_h)|\bm\phi|_h) + 2\alpha \tau d^*(d \dot f + \eta_{\dot u}\lrcorner i F_h).
\end{split}
\end{equation} 
Here $\operatorname{P}^*\operatorname{P}$ is, up to a multiplicative
constant factor, the Lichnerowicz operator of the K\"ahler manifold
$(\PP^1,\omega)$, $J$ is the standard almost complex structure on $\PP^1$, $A$ is the Chern connection of $h$, and $\eta_{\dot u}$ denotes the $\omega$-Hamiltonian vector field associated to $\dot u$. Moreover, the operator $\mathbf{L}_{\alpha, \varepsilon}$ satisfies (see \cite[Lemma 6.3]{AlGaGaPi})
\begin{equation}\label{eq:Lsa}
\begin{split}
\langle  (4\alpha \dot f, \dot u),\mathbf{L}_{\alpha, \varepsilon} (\dot u, \dot f) \rangle_{L^2} & = \|L_{\eta_{\dot u}}J\|_{L^2}^2 + 4\alpha  \|d \dot f + \eta_{\dot u}\lrcorner i F_h\|_{L^2}^2 + 4\alpha\|J \eta_{\dot u}\lrcorner d_A \bm\phi - \dot f \bm\phi \|^2_{L^2}\\
& + 4\alpha \langle (J \eta_{\dot u} \lrcorner (i d \dot f + \eta_{\dot u}\lrcorner F_h),\mathbf{T}_{\alpha, \varepsilon}^0(u,f) \rangle_{L^2},
\end{split}
\end{equation}
where the $L^2$ products are taken with respect to the fixed metric $\omega$. Assume now that $(\omega_0,h_0)$ is a smooth solution of \eqref{eq:continuitypath} for $\alpha_0 \in (0,\tfrac{1}{\tau N}]$. Then, the linearization of $\mathbf{T}_{\alpha, \varepsilon}$ at $(0,0)$ for $(\alpha, \varepsilon) = (\alpha_0,1)$ equals
$$
\delta \mathbf{T}_{\alpha_0,1} = \delta_{|(0,0)} \mathbf{T}_{\alpha_0,1} = \mathbf{L}_{\alpha_0,1}(\varepsilon \cdot , \cdot)
$$ 
and the operator $\mathbf{L}_{\alpha_0,1}$ is self-adjoint. To characterize the kernel of $\mathbf{L}_{\alpha_0,1}$, let $\Aut(\PP^1,L,\bm\phi)$ denote the group of automorphisms of $(\PP^1,L,\bm\phi)$, given by automorphisms of $L$ covering an element in $\Aut(\PP^1)= \operatorname{PGL}(2,\CC)$ and preserving $\bm\phi$. Then, formula \eqref{eq:Lsa} implies that (see \cite[Lemma 6.3]{AlGaGaPi})
$$
A^\perp \eta_{\dot u} + i \dot f \mathbf{1} \in \operatorname{Lie} \Aut(\PP^1,L,\bm\phi)
$$
for any  $(\dot u, \dot f) \in \operatorname{ker} \mathbf{L}_{\alpha_0,1}$. Here, $A^\perp \eta_{\dot u}$ denotes the horizontal lift of $\eta_{\dot v}$ to the total space of $L$ using the Chern connection $A$ of $h_0$ and $\mathbf{1}$ denotes the canonical vertical vector field on $L$. Now, by assumption, $\bm\phi$ vanishes at more than two points and therefore element in $\Aut(\PP^1,L,\bm\phi)$ must project to the identity in $\operatorname{PGL}(2,\CC)$. Since $\bm\phi\neq 0$, it follows that $ \Aut(\PP^1,L,\bm\phi) = \{1\}$ (see \cite[Section 4.1]{AlGaGaPi}), and therefore
$$
\operatorname{ker} \mathbf{L}_{\alpha_0,1} = \RR \times \{0\}.
$$
The openness of $S$ now follows by application of the implicit function theorem in a Sobolev completion of $(C^\infty(X)/\RR) \times C^\infty(X)$. Smoothness of the solution follows easily bootstrapping in \eqref{eq:continuitypath}. To prove that $S$ is non-empty, we simply apply Theorem \ref{thm:HSY}: at $\alpha = \tfrac{1}{\tau N}$ the equations~\eqref{eq:continuitypath} admit a smooth solution $(\omega_0,h_0)$ with $\Vol_{\omega_0} \gg \tfrac{4\pi N}{\tau}$. 
\end{proof}

\subsection{Futaki invariant and extremal pairs}\label{subsec:Futaki}

In the strictly polystable case, that is, for $D=\frac{N}{2}p_1 + \frac{N}{2}p_2$, Theorem \ref{th:Yang} implies that the solution at $\alpha = \tfrac{1}{\tau N}$ has an $T^2$-symmetry, and hence the proof of Lemma \ref{lem:openst} does not apply. In order to prove an analogue of Lemma \ref{lem:openst} in this case in Section \ref{subsec:openS1}, we need two theoretical devices: the $\alpha$-\emph{Futaki invariant} for the gravitating vortex equations,  introduced in \cite{AlGaGaPi}, and a notion of \emph{extremal pair}. These are analogues for the gravitating vortex equations of the familiar notions of Futaki invariant and extremal metric in K\"ahler geometry, respectively. 

To start, let us consider the situation that $\bm\phi \in H^0(\PP^1,L)$ vanishes at exactly two points. We make the identification $L=\cO_{\PP^1}(N)$, with $N\defeq
c_1(L)>0$, and fix homogeneous coordinates $[x_0,x_1]$ on
$\PP^1$ such that
\begin{equation}\label{eq:phi-Einstein-Bogomonyi}
\bm\phi\cong x_0^{N- \ell}x_1^\ell,
\end{equation}
with $0< \ell<N$ (the case $\ell= N/2$ corresponds to the strictly polystable case in Theorem \ref{th:Yang}). Here, we identify $H^0(\mathbb{P}^1,L)\cong S^N(\CC^2)^*$ with the space of degree
$N$ homogeneous polynomials in the coordinates $x_0,x_1$, so it is a
$\operatorname{GL}(2,\CC)$-representation, where $g\in\GL(2,\CC)$ maps
a polynomial $p(x_0,x_1)$ into the polynomial
$p(g^{-1}(x_0,x_1))$. Denote by $\rho$ the canonical $\GL(2,\CC)$-linearization of
$L$ 
induced by the $\operatorname{GL}(2,\CC)$-representation $H^0(\PP^1,L)$. Note that an element in the centre, $\lambda\in\CC^*\subset\operatorname{GL}(2,\CC)$, acts via $\rho$ on $L$ by fibrewise multiplication by $\lambda^{-N}$.

\begin{lemma}[\cite{AlGaGaPi}]\label{lem:Yangconjextended}
Let $\bm\phi \in  H^0(\mathbb{P}^1,L)$ as in \eqref{eq:phi-Einstein-Bogomonyi}. Then $\Aut(\PP^1,L,\bm\phi)$ is
given by the image of the standard maximal torus $\CC^* \times \CC^*
\subset \operatorname{GL}(2,\CC)$ under the morphism
$$
\rho_\ell \colon \CC^* \times \CC^* \to \Aut(\PP^1,L)
$$
defined by
$$
\rho_\ell(\lambda_0,\lambda_1) = \lambda_0^{N- \ell} \lambda_1^{\ell}\rho(\lambda_0,\lambda_1)
$$
where $\lambda_0^{N- \ell} \lambda_1^{\ell}$ acts on $L$ by multiplication on the fibres.
\end{lemma}

By the Matsushima-Lichnerowicz type theorem for the gravitating vortex equations (see \cite[Theorem 3.6]{AlGaGa3}), the group of isometries of the solution in Theorem \ref{th:Yang} for $\ell = \tfrac{N}{2}$ corresponds to the maximal compact (see Lemma \ref{lem:Yangconjextended})
$$
K = S^1 \times S^1 \subset \Aut(\PP^1,L,\bm\phi).
$$
Note here that the projection of $K$ onto $\Aut(\PP^1) = \operatorname{PGL}(2,\CC)$ is a circle.

To introduce the $\alpha$-Futaki invariant, we fix $\alpha > 0$ and $\tau > 2N$. Denote by $B$ the space of pairs $(\omega,h)$ consisting of a K\"ahler form $\omega$ on $\PP^1$ with volume $2\pi$, and a Hermitian metric $h$ on $L$. Define a map 
\begin{equation}\label{eq:futakigravvort.arrow}
\cF_{\alpha,\tau}\colon\Lie\Aut(\PP^1,L,\bm\phi)\lto\CC,
\end{equation}
by the following formula, for all $y\in\Lie\Aut(\PP^1,L,\bm\phi)$, where
$(\omega,h)\in B$:
\begin{equation}\label{eq:futakigravvort}
\begin{split}
\langle\cF_{\alpha,\tau},y\rangle & = 4i\alpha\int_{\PP^1} A_h y \(i\Lambda_\omega F_h + \frac{1}{2}|\bm\phi|^2_h - \frac{\tau}{2}\)\omega
- \int_{\PP^1}  \varphi \(S_\omega + \alpha \Delta_\omega |\bm\phi|_h^2 - 2i \alpha\tau \Lambda_\omega F_h\)\omega.
\end{split}
\end{equation}
Here, $A_h$ is the Chern connection of $h$ on $L$, $A_hy\in C^\infty(\PP^1,i\RR)$ is the vertical projection of $y$ with respect to $A_h$, and the complex valued function
$\varphi$ on $\PP^1$ is defined as follows. Let $\check{y}$ be the
holomorphic vector field on $\PP^1$ covered by $y$ and $A_h^\perp\check{y}$ its
horizontal lift to a vector field on the total space of $L$ given by
the connection $A_h$. Therefore, $y$ has a decomposition
\begin{equation}\label{eq:holvectfield-Vert+Horiz.2}
y = A_h y+A_h^\perp\check{y}  
\end{equation}
into its vertical and horizontal components. Then
$\varphi\defeq\varphi_1+i\varphi_2\in C^\infty(\PP^1,\CC)$ is
determined by the unique decomposition,
\begin{equation}\label{eq:vect-field-decomposition.2}
\check{y}=\eta_{\varphi_1}+J\eta_{\varphi_2} 
\end{equation}
associated to the K\"ahler form $\omega$ (see~\cite{LS1}), where
$\eta_{\varphi_j}$ is the Hamiltonian vector field of the function
$\varphi_j\in C^\infty(\PP^1)$ (here we assume $\int_{\PP^1} \varphi_j \omega = 0$), for $j=1,2$, and $J$ is the almost complex structure of $\PP^1$. 
Note that the previous decomposition uses the fact that $\PP^1$ is simply connected.

The non-vanishing of $\cF_{\alpha,\tau}$ provides an obstruction to
the existence of gravitating vortices.

\begin{proposition}[\cite{AlGaGaPi}]\label{prop:futakibis}
The map $\cF_{\alpha,\tau}$ is independent of the choice of 
$(\omega,h)\in B$. It is a character of the Lie algebra
$\Lie\Aut(\PP^1,L,\bm\phi)$, that vanishes identically if there exists a
solution of the gravitating vortex equations
\eqref{eq:gravvortexeq1} on $(\PP^1,L,\bm\phi)$ with coupling constant $\alpha$, symmetry breaking parameter $\tau$, and volume $2\pi$.
\end{proposition}

Next, we introduce the relevant notion of extremal pair. Let $\omega$ be a K\"ahler form on $\PP^1$ and $h$ a Hermitian metric on $L$. Associated with the pair $(\omega,h)$, we consider a vector field
\[
\zeta_{\alpha,\tau}(\omega,h)\defeq i(i\Lambda_\omega F_h + \frac{1}{2}|\bm\phi|^2_h - \frac{\tau}{2})\mathbf{1} + A_h^\perp \eta_{\alpha,\tau}
\]
on the total space of $L$, where $\mathbf{1}$ denotes the canonical vertical vector field on $L$ and $\eta_{\alpha,\tau}$ is the Hamiltonian
vector field of the smooth function
\begin{equation}\label{eq:Salphatau}
S_\omega+\alpha \Delta_\omega|\bm\phi|^2_h-2\alpha\tau i\Lambda_\omega F_h.
\end{equation}
Note that the vector field $\zeta_{a,\alpha,\tau}(\omega,h)$ is $\CC^*$-invariant (actually it belongs to the extended gauge group determined by $(\omega,h)$, in the sense of \cite{AlGaGaPi}).

\begin{definition}\label{def:extremal_pair}
The pair $(\omega,h)$ is \emph{extremal} if
\[
\zeta_{\alpha,\tau}(\omega,h)\in\Lie\Aut(\PP^1,L,\bm\phi),
\]
that is, the vector field $\zeta_{\alpha,\tau}(\omega,h)$ is
holomorphic and preserves $\bm\phi$.
\end{definition}

Of course, solutions of the gravitating vortex equations \eqref{eq:gravvortexeq1} correspond, precisely, to extremal pairs $(\omega,h)$ such that $\zeta_{\alpha,\tau}(\omega,h) = 0$. More generally, the existence of an extremal pair with $\zeta_{\alpha\,\tau}(\omega,h)\neq 0$ is an obstruction to the existence of solutions of the gravitating vortex equations. This follows from Proposition~\ref{prop:futakibis},
because $\zeta_{\alpha\,\tau}(\omega,h)\neq 0$ implies
\begin{equation}\label{eq:ineqFut}
\langle\cF_{\alpha,\tau},\zeta_{\alpha,\tau}(\omega,h)\rangle<0,
\end{equation}
as can be shown by applying formula~\eqref{eq:futakigravvort} using $(\omega,h)$
to $y=\zeta_{\alpha,\tau}(\omega,h)$ (cf.~\cite[Proposition 4.2]{AlGaGa}). The upshot of the previous abstract discussion is the following useful result.

\begin{proposition}\label{prop:futakibis}
Assume that the divisor $D = (N - \ell)p_1 + \ell p_2$ induced by $\bm\phi$ is strictly polystable, that is, $\ell = \tfrac{N}{2}$. Then, $\cF_{\alpha,\tau} = 0$ and, consequently, any extremal pair $(\omega,h)$ is a solution of the gravitating vortex equations \eqref{eq:gravvortexeq1}.
\end{proposition}

\begin{proof}
The vanishing of $\cF_{\alpha,\tau}$ follows from \cite[Lemma 4.6]{AlGaGaPi}, which implies that
$$
\langle \cF_{\alpha,\tau},y\rangle = 2\pi i \alpha (2N - \tau) (2 \ell - N),
$$
where
\begin{equation}\label{eq:y}
y = \( \begin{array}{cc}
0 & 0 \\
0 & 1
\end{array} \) \in \Lie\Aut(\PP^1,L,\bm\phi) \subset \mathfrak{gl}(2,\CC)
\end{equation}
(and similarly for the other generator). The statement now follows from the fact that, if $(\omega,h)$ is an extremal pair with $\zeta_{\alpha\,\tau}(\omega,h)\neq 0$, then \eqref{eq:ineqFut} holds.
\end{proof}

To finish this section, we give an example of extremal pair which is not a solution of the gravitating vortex equations.

\begin{example}
Consider $L = \mathcal{O}_{\PP^1}(1)$ on $\PP^1$, with $\bm\phi = x_0$ (see \eqref{eq:phi-Einstein-Bogomonyi}). Let $\omega_{FS}$ be the Fubini--Study K\"ahler metric on $\PP^1$, normalized so that $\int_{\PP^1}\omega_{FS} = 2\pi$. Consider also the Fubini--Study Hermitian metric $h_{FS}$ on $\mathcal{O}_{\PP^1}(1)$. We choose coordinates $z = \frac{x_1}{x_0}$,
so that $\bm\phi = 1$ and
$$
\omega_{FS} = \frac{i dz \wedge d\overline{z}}{(1 + |z|^2)^2}, \qquad h_{FS} = \frac{1}{1 + |z|^2} = |\bm\phi|_{h_{FS}}^2.
$$
Using now that $i\Lambda_{\omega_{FS}}F_{h_{FS}}$ and $S_{\omega_{FS}}$ are constant, the Hamiltonian vector field corresponding to \eqref{eq:Salphatau}, with $\omega = \omega_{FS}$, equals the Hamiltonian vector field of the smooth function 
$$
\alpha\Delta_{\omega_{FS}} |\bm\phi|_{h_{FS}}^2 = 2\alpha \frac{1- |z|^2}{1 + |z|^2},
$$
which turns out to be
$$
v = 4\alpha iz\frac{\partial}{\partial z}.
$$
Using now the equalities
$$
\dbar (i(i\Lambda_{\omega_{FS}} F_{h_{FS}} + \frac{1}{2}|\bm\phi|^2_{h_{FS}} - \frac{\tau}{2})) = \frac{i}{2} \dbar |\bm\phi|^2_{h_{FS}} = \frac{i}{8\alpha}i_{\check v^{1,0}}\omega_{FS} = - \frac{1}{8\alpha}i_{v^{1,0}}F_{h_{FS}},
$$
it follows from \cite[Lemma 4.1]{AlGaGaPi} that $(\omega_{FS},h_{FS})$ is an extremal pair for $(\PP^1,\mathcal{O}_{\PP^1}(1),x_0)$, provided that $\alpha = 1/8$. Taking $2 < \tau < 8$, we have $\alpha = 1/8 \in (0,\tfrac{1}{\tau})$, and therefore in this case $c > 0$.
\end{example}

\begin{remark}\label{rem:extremal_pair}
One can compare the definition of extremal pair for the K\"ahler--Yang--Mills equations in \cite[Definition 4.1]{AlGaGa} with Definition~\ref{def:extremal_pair} via the process of dimensional reduction described in~\cite[Section
3.2]{AlGaGa2}. Under this comparison, the former definition corresponds
to the latter only for $\alpha=1/4$, but clearly the notion of extremal pair for the K\"ahler--Yang--Mills equations can be generalized by considering a modification of the vector field $\zeta_\alpha$ (see~\cite[(4.136)]{AlGaGa}), with the
Hermite--Yang--Mills term multiplied by $a\in\RR_{>0}$.
\end{remark}

\subsection{The strictly polystable case}\label{subsec:openS1}

We address now the analogue of Lemma \ref{lem:openst} in the case that $\bm\phi$ vanishes at exactly two points, with multiplicity $\frac{N}{2}$. The proof will follow by adaptation of the Lebrun-Simanca argument for extremal K\"ahler metrics \cite{LS1} (cf. \cite{AlGaGa}). By part $(1)$ of Theorem \ref{thm{existence}}, it will suffice for our applications to assume that the K\"ahler class of the solutions is fixed to some particular value. 

Let $(\omega_0,h_0)$ be a solution of the gravitating vortex equations \eqref{eq:gravvortexeq1} on $(\PP^1,L,\bm\phi)$ with coupling constant $\alpha_0 \in (0,\tfrac{1}{\tau N}]$. As already mentioned in Section \ref{subsec:Futaki}, the group of isometries of the solution corresponds to the maximal compact (see Lemma \ref{lem:Yangconjextended})
$$
K = S^1 \times S^1 \subset \CC^* \times \CC^* = \Aut(\PP^1,L,\bm\phi).
$$
Note that the projection of $K$ onto $\operatorname{PGL}(2,\CC)$ is a circle. Given $k > 1$ a non-negative integer, denote by $L^2_k(\PP^1)$ the Hilbert space of $L^2$-functions on $\PP^1$ with $k$ distributional derivatives in $L^2$, and let
\begin{equation}
\label{eq:inv-Sobolev-spc}
L^2_k(\PP^1)^{K} \subset L^2_k(\PP^1)
\end{equation}
be the closed subspaces of $K$-invariant functions. Given $\alpha\in\RR$, the maps $\mathbf{T}_\alpha$ and $\mathbf{L}_{\alpha} = \mathbf{L}_{\alpha,0,0}$, defined as in the proof of Lemma~\ref{lem:openst} setting $\varepsilon = 1$, induce well-defined maps:
\begin{equation}\label{eq:Talphamap-Lalpha-invariant}
\begin{split}
\hat{\operatorname{\mathbf{T}}}_\alpha \colon &
\mathcal{V} \lto L^2_{k}(\PP^1)^{K} \times L^2_{k+2}(\PP^1)^{K},
\\
\hat{\mathbf{L}}_{\alpha} \colon &
L^2_{k+4}(\PP^1)^{K} \times L^2_{k+4}(\PP^1)^{K} \lto L^2_{k}(\PP^1)^{K} \times L^2_{k+2}(\PP^1)^{K},
\end{split}
\end{equation}
where 
\begin{equation}
\label{eq:subset-V}
\mathcal{V} \subset L^2_{k+4}(\PP^1)^{K} \times L^2_{k+4}(\PP^1)^{K}.
\end{equation}
is a neighborhood of the origin and $\hat{\operatorname{\mathbf{T}}}_\alpha$ is $C^1$ with Fr\'echet derivative at $\alpha = \alpha_0$ given by $\hat{\mathbf{L}}_{\alpha_0}$.

Let $d^*$ and $\mathbf{G}$ be the formal adjoint of the de Rham
differential and the Green operator of the Laplacian for the fixed
metric $\omega_0(\cdot,J_0\cdot)$, respectively. Then for any symplectic
form $\omega$ and any $\eta$ in the Lie algebra
$\LieH_{\omega}$ of Hamiltonian vector fields over
$(\PP^1,\omega)$ we have
\begin{equation}
\label{eq:Hamiltonianfunction=Greenfunction}
d(\mathbf{G}d^*(\eta\lrcorner\omega)) = \eta\lrcorner \omega.
\end{equation}
As the image of the Green operator is perpendicular to the constants,
the Hamiltonian function $f =
\mathbf{G}d^*( \eta \lrcorner \omega)$ is
`normalized' for the volume form $\omega_0$, that is, $\int_X f
\omega_0 =0$.

For each $(u,f)\in\cV$, we define a linear map
\begin{equation}
\label{eq:Pbold}
\begin{gathered}
\xymatrix @R=0ex @C=-1ex {
\!\!\!\!\mathbf{P}_{(u,f)}=(\mathbf{P}^0_{u},\mathbf{P}^1_{f})\colon
& \RR \times \Lie K \ar[r] & **[r]
L^2_{k}(\PP^1)^{K} \times L^2_{k+2}(\PP^1)^{K}\\
& (t,v) \ar@{|->}[r] &  **[r]
\(\mathbf{G}d^*(p(v) \lrcorner \omega) + t,\theta_{h}v\),\\
}
\end{gathered}
\end{equation}
where $p \colon
K \to \operatorname{PGL}(2,\CC)$ is the natural projection, while
$$
(\omega,h) = (\omega_0 + dd^c u,e^{2f}h_0).
$$
The map $\mathbf{P}_{(u,f)}$ attaches to a vector field
$v \in \Lie K$ its vertical part $\theta_{h}v$, calculates the normalized Hamiltonian
function of the vector field $p(v)$ over $(X,\omega)$, and
adds an extra parameter $t$ which accounts for the fact that
Hamiltonian functions are only determined up to a constant.

Here is the key link between extremal pairs and the linearization of
the gravitating vortex equations. The proof is analogue to the proof of \cite[Lemma 4.8]{AlGaGa}, and it is therefore omitted.

\begin{lemma}
\label{lemma:PTalpha}
Let $(u,f) \in \mathcal{V}$.
\begin{enumerate}
\item[\textup{(1)}]
   $\mathbf{P}_{(u,f)}$ is injective.
\item[\textup{(2)}]
  If $\hat{\operatorname{\mathbf{T}}}_\alpha(u,f)\in\Im\mathbf{P}_{(u,f)}$,
  then $(\omega,h)$ is an extremal pair.
\item[\textup{(3)}]
  $\Im\mathbf{P}_{0}\subset\ker\hat{\mathbf{L}}_{\alpha}$, with
  equality if $h_0$ is a solution of the vortex equation \eqref{eq:vortexeq} with respect to $\omega_0$.
\end{enumerate}
\end{lemma}

Let $\langle\cdot,\cdot\rangle_{\omega_0}$ be the $L^2$-inner product on $
L^2_{k}(\PP^1)^{K} \times L^2_{k+2}(\PP^1)^{K}$. Arguing as in \cite[Section 4.4]{AlGaGa} it is easy to see that the orthogonal
projectors onto $\Im\mathbf{P}_{(u,f)}$, denoted
\[
\Pi_{(u,f)} \colon L^2_{k}(\PP^1)^{K} \times L^2_{k+2}(\PP^1)^{K}
\lto L^2_{k}(\PP^1)^{K} \times L^2_{k+2}(\PP^1)^{K},
\]
vary smoothly with $(u,f) \in \mathcal{V}$ and, furthermore, the origin has an open neighbourhood $\mathcal{V}_0 \subset \mathcal{V}$ such that the following holds
(cf.~\cite[(5.3)]{LS1}):
\begin{equation}
\label{eq:P0Pphi}
\ker(\Id - \Pi_{(u,f)}) = \ker(\Id - \Pi_0)\circ(\Id - \Pi_{(u,f)}).
\end{equation}

For any pair of non-negative integers $(l,m)$, let $I_{l,m}\subset
L^2_{l}(\PP^1)^{K} \times L^2_{m}(\PP^1)^{K}$ be the orthogonal
complement of $\Im\mathbf{P}_{0}$. Define
\[
\mathcal{W} = \mathcal{V}_0 \cap I_{k+4,k+4}.
\]
Note that, under the assumptions in the last part of
Lemma~\ref{lemma:PTalpha}, the subspace $\mathcal{W}$ is perpendicular
to $\ker\mathbf{L}_{\alpha}$. Define a LeBrun--Simanca map~\cite[\S 5]{LS1}
\begin{equation}
\label{eq:Tmap}
\begin{gathered}
\xymatrix @R=0ex @C=-4ex {
**[l] \mathbf{T}_\alpha \colon & \mathcal{W} \ar[r] & **[r] I_{k,k+2}\\
& (u,f) \ar@{|->}[r] &  **[r]
(\Id - \Pi_0)\circ(\Id - \Pi_{(u,f)})\circ \hat{\operatorname{\mathbf{T}}}_\alpha(u,f).\\
}
\end{gathered}
\end{equation}

Given $(\dot{u},\dot{f}) \in I_{k+4,k+4}$, the directional derivative of $\mathbf{T}_{\alpha_0}$ at the origin in
the direction $(\dot{u},\dot{f})$ is
\begin{equation}
\label{eq:Talphamapder2}
\delta_0\mathbf{T}_{\alpha_0} (\dot{u},\dot{f})=(\Id - \Pi_0) \circ \hat{\mathbf{L}}_{\alpha_0}(\dot{u},\dot{f}).
\end{equation}
where we have used that $(\omega_0,h_0)$ is a solution of the gravitating vortex equations.

We can now prove the main result of the present section. Recall that the section $\bm\phi$ determines an effective divisor $D$ on $\mathbb{P}^1$.

\begin{lemma}\label{lem:openpolyst}
Assume that $D$ is GIT strictly polystable. Then, the following set is open 
\begin{equation*}
S = \{\alpha \in (0,\tfrac{1}{\tau N}]  \; \textrm{such that \eqref{eq:continuitypath} has a smooth solution }(\omega, h) \text{ with }\Vol_\omega = \Vol_{\omega_0} \}.
\end{equation*}
\end{lemma}

\begin{proof}
Let $(\omega_0,h_0)$ be a solution of the gravitating vortex equations \eqref{eq:gravvortexeq1} on $(\PP^1,L,\bm\phi)$ with coupling constant $\alpha_0$. 
Since the map
$\mathbf{T}_\alpha$ depends linearly on $\alpha$, it can be viewed as a $C^1$ map
$\mathbf{T}\colon \RR^2 \times \mathcal{W} \to I_{k,k+2}$, whose the
Fr\'echet derivative at the origin with respect to $u$ and $f$ for $\alpha = \alpha_0$ is
$\delta_0\operatorname{\mathbf{T}}_{\alpha_0} = (\Id-\Pi_0)\circ
\hat{\mathbf{L}}_{\alpha_0}$, by~\eqref{eq:Talphamapder2}. Since $h_0$
solves the vortex equation \eqref{eq:vortexeq} with respect to $\omega_0$,
Lemma~\ref{lemma:PTalpha} applies and $(\Id - \Pi_0)\circ
\hat{\mathbf{L}}_{\alpha_0}$ is an isomorphism. Therefore, by the
implicit function theorem, there exists an open neighbourhood $U$ of $\alpha_0$ such
that for all $\alpha \in U$ there exists a pair
$(u,f) \in I_{k+4,k+4}$ such that
\[
\hat{\operatorname{\mathbf{T}}}_{\alpha}(u,f)\in\ker \; (\Id-\Pi_0)\circ (\Id-\Pi_{u,f}),
\]
and therefore 
$\hat{\operatorname{\mathbf{T}}}_{\alpha}(u,f)\in\Im\mathbf{P}_{(u,f)}$ by~\eqref{eq:P0Pphi}. Hence the pair $(\omega,h)$ determined by $(u,f)$
is extremal with coupling constant $\alpha$, by Lemma~\ref{lemma:PTalpha}. Finally, $(\omega,h)$ is a solution of the gravitating vortex equations by Proposition \ref{prop:futakibis} (smoothness follows by bootstrapping in \eqref{eq:gravvortexeq1}).
\end{proof}

\section{A priori estimates and Closedness}\label{sec:close}

The goal of this section is to prove that the existence of solutions of the gravitating vortex equations with $c > 0$ is a closed condition for the coupling constant, concluding the proof of Theorem \ref{th:main}.  Every K\"ahler metric along this continuity path lies in the fixed  K\"ahler class $[\omega_0]$. For this, in the present section we shift to the Riemannian point of view in Section \ref{subsec:RGV}, using systematically Lemma \ref{lem:KR}. We will assume that Theorem \ref{thm{existence}} holds and, without loss of generality, we set $V = 2\pi$ in order to simplify the notation.

\subsection{Scalar curvature and state function estimate}\label{subsec:apriori}

In this section we establish \emph{a priori} estimates for the state function and the scalar curvature of a solution of \eqref{eq:RGV}. To this end, we first prove two basic estimates about the state function $\Phi$.

\begin{lemma}\label{lem:statebound}
Let $(g,\eta,\Phi)$ be a smooth solution of \eqref{eq:RGV}. Then
\hfill
\begin{itemize}\label{lem:bound-for-Phi}
		\item $0\leqslant\Phi \leqslant\tau$
		\\
		\item $\frac{1}{2\pi}\int_{S^2} \Phi \vol_g = \tau - 2N$
   
	\end{itemize}
\end{lemma}

\begin{proof}
	The first item follows from \cite[Lem. 6.5]{AlGaGaPi} combined with Lemma \ref{lem:KR}, and it is also transparent from the maximum principle applied to the third equation in \eqref{eq:RGV}. The second item follows immediately from integrating the first equation in \eqref{eq:RGV}. 
\end{proof}



The following result is key for our method of proof for Theorem \ref{th:main}. In the case $c > 0$, which is most relevant for the present work, it implies that the K\"ahler metric of the gravitating vortices must have positive curvature.

\begin{proposition}\label{lower-bound-on-scalar}
	For any solution $(\eta, g, \Phi)$ to the system \eqref{eq:RGV}, the curvature of the metric is bounded from below: $S_g\geqslant c$, where $c$ is the topological constant defined in \eqref{eq:constantc}. As a consequence, the first eigenvalue $\lambda_1(\Delta_g)\geqslant c$. 
\end{proposition}

\begin{proof}
	It is derived by simply combining the first and third equations in \eqref{eq:RGV}. Since 
	\[
	\tau-\Phi = \Delta_g \log \Phi 
	=
	\frac{\Delta_g \Phi}{\Phi} 
	+
	 \frac{|\nabla \Phi|^2}{\Phi^2}
	\]
	away from the vanishing points $\{p_j\}$ of $\Phi$, we have 
	\begin{equation}\label{Scalar-curvature-formula}
	\begin{split}
	S_g
	& =
	c+\alpha\tau(\tau-\Phi) - \alpha \Delta_g \Phi\\
	& =
	c+ \alpha \frac{|\nabla\Phi|^2}{\Phi} + \alpha\tau(\tau-\Phi)- \alpha\Phi(\tau-\Phi)\\
	& =
	c+\alpha \frac{|\nabla\Phi|^2}{\Phi} +\alpha(\tau-\Phi)^2
	\end{split}
	\end{equation}
	
Because the terms $S_g, c, \alpha(\tau-\Phi)^2$ are smooth functions on $S^2$, it follows the term $\frac{|\nabla \Phi|^2}{\Phi}$ is also smooth (and obviously nonnegative) on $S^2$. The proposition follows. 
\end{proof}

The proof of the previous result singles out the term $\frac{|\nabla \Phi|^2}{\Phi}$. Our next task is to obtain an estimate for this quantity. For this, it will be useful to have a more geometric characterization in terms of the original unknowns for the gravitating vortex equations (see Lemma \ref{lem:KR}).

\begin{lemma}\label{lem:Phiphi}
With the notation of Lemma \ref{lem:KR}
$$
\frac{|\nabla \Phi|^2}{\Phi} = 2|\nabla^{1,0}\bm\phi|^2.
$$
\end{lemma}
\begin{proof}
Choose a local holomorphic coordinate $z$ and a local holomorphic frame $\bm e$ of $L$, so that $\bm\phi=f\bm e$. Then, $\Phi=| f |^2 h$ where $h=h(\bm e, \bm e)$ is the Hermitian metric. Then $\Phi_z
=\bar f h \left( f_z + f(\log h)_z \right)$, thus 
\[
|\nabla\Phi|^2
=
2|f|^2 h^2 \left| f_z + f(\log h)_z\right|^2 g^{z\bar z}
=
2\Phi |\nabla^{1,0}\bm\phi|^2.
\]
\end{proof}

It is interesting to observe that the quantity $|\nabla^{1,0}\bm\phi|^2$ appears very naturally in the moment map interpretation of the gravitating vortex equations, for the action of the extended gauge group (see \cite[Formula (3.8)]{AlGaGaPi}). In fact, the infinite-dimensional symplectic geometry singles out 
$$
S_\omega - 2 \alpha |\nabla^{1,0}\bm\phi|^2_h - \alpha \tau i \Lambda_\omega F_h = c',
$$
with $c' \in \RR$, as the second equation in \eqref{eq:gravvortexeq1} (of course, the corresponding systems of PDE are equivalent). In the physical case $c = 0$, the function $|\nabla^{1,0}\bm\phi|^2$ appears in the trace of the stress-energy tensor of the Einstein field equations (see \cite{Yang}).

To obtain an estimate for $|\nabla^{1,0}\bm\phi|$, let us recall a general formula from K\"ahler geometry. For the benefit of the reader, a short proof is given.

\begin{proposition}[Weitzenb\"ock Formula]\label{Formula-Kahler}
Let $(M, \omega)$ be a K\"ahler manifold, and $L\to M$ be a holomorphic line bundle equipped with a Hermitian metric $h$ and a holomorphic section $s$, then $\nabla s=\nabla^{1,0}s$ satisfies the following equality: 

	\begin{align*}
	- \Delta |\nabla s|^2 
	& = |\nabla \nabla s|^2 + |F_h|^2 |s|^2 
	+ (\text{Ric}_\omega - 2i F_h)(\nabla s, \nabla s)
	- |\nabla s|^2 \Lambda_\omega iF_h\\
	& - 2 \text{Re}\;\langle \nabla (\Lambda_\omega iF_h),\nabla |s|^2\rangle 
	\end{align*}
\end{proposition}
\begin{proof}
	Let $\bm e$ be a local holomorphic frame of $L$, and let $s=f\bm e$, suppose $h(\bm e, \bm e)=h$, then $|s|^2 = |f|^2 h$, and 
	$\nabla s=(\partial f  + f h^{-1}\partial h) \bm e$, thus
	\[
	|\nabla s|^2 
	= f_i \bar f_{\bar j}g^{i\bar j}h + \bar f f_i h_{\bar j}g^{i\bar j}
	+ f\bar f_{\bar j}h_i g^{i\bar j} + |f|^2 h^{-1}h_i h_{\bar j}g^{i\bar j}
	\]
	Using $K$-coordinates for the closed $(1,1)$-form $iF_h$ illustrated in \cite[Prop. 2.2]{Ruan} for simplicity of calculation at a point, we get 
	\begin{equation}
		\begin{split}
			-\Delta |\nabla s|^2 
			& = 
			f_{ik} \overline{f_{jl}}g^{i\bar j}h g^{k\bar l}
			+ f_i\overline{f_j} g^{i\bar j}_{\phantom{i \bar j},k\bar l} g^{k\bar l}h 
			+ f_i \overline{f_j} g^{i\bar j}h_{k\bar l}g^{k\bar l}
			+ f_i\overline{f_l} h_{k\bar j}g^{i\bar j}g^{k\bar l}\\
			& + \bar f f_i h_{k\bar j\bar l}g^{i\bar j}g^{k\bar l}
			+ f_k \overline{f_j}h_{i\bar l}g^{i\bar j}g^{k\bar l}
			+ f \overline{f_j}h_{i\bar l k}g^{i\bar j}g^{k\bar l}
			+ |f|^2 h^{-1}h_{i\bar l}h_{k\bar j}g^{i\bar j}g^{k\bar l}
			\end{split}
		\end{equation}
	Denote $iF_h =i\Theta_{k\bar l}dz_i\wedge d\bar z_j$, we know $h_{k\bar l}=- h\Theta_{k\bar l}$ and $h_{k\bar l\bar j}=-h\Theta_{k\bar l \bar j}$, and $g^{k\bar l}\Theta_{k\bar l \bar j}=(\Lambda_\omega iF_h)_{\bar j}$. Plugging those terms in the above expression yields the formula in the proposition.
	\end{proof}

Assume now that $(g,\eta,\Phi)$ is a smooth solution of \eqref{eq:RGV}, and take $(\omega,h)$ as in Lemma \ref{lem:KR}. Then, using that $\text{Ric}_\omega=S_g\omega$, $s=\bm\phi$ and $iF_h =\eta = \frac{1}{2}(\tau-\Phi)\omega$, we derive the following:

\begin{proposition}	\label{Hormander-estimate}
Given a smooth solution $(\eta, g, \Phi)$ of the system \eqref{eq:RGV}, 
	\begin{equation*}
	- \Delta_g |\nabla^{1,0}\bm\phi|^2 
	=
	|\nabla\nabla\bm\phi|^2 
	+ 
	\frac{1}{4}\Phi (\tau-\Phi)^2
	+
	\left(  2 \alpha|\nabla^{1,0}\bm\phi|^2 + c + (\tau-\Phi)\left[  
	\alpha(\tau- \Phi) - \frac{3}{2}
	\right] 
	+
	\Phi
	\right) |\nabla^{1,0}\bm\phi|^2 
	\end{equation*}
	
\end{proposition}

\begin{proof}
	Directly plugging in equation \eqref{Formula-Kahler}, we get
	\begin{equation}
	- \Delta_g |\nabla^{1,0}\bm\phi|^2 
	=
	|\nabla\nabla\phi|^2 
	+ 
	\frac{1}{4}\Phi (\tau-\Phi)^2
	+
	\left( S_g -(\tau-\Phi)  \right) |\nabla^{1,0}\bm\phi|^2 
	-
	\frac{1}{2}(\tau-\Phi) |\nabla^{1,0}\bm\phi|^2 
	+
	|\nabla \Phi|^2
	\end{equation}
	The proposition follows by replacing $S_g$ using formula \eqref{Scalar-curvature-formula}.
\end{proof}

We are ready to prove our estimate for $\frac{|\nabla \Phi|^2}{\Phi}$.

\begin{corollary}[derivative estimate]\label{derivative-estimate}
For any solution $(\eta, g, \Phi)$ of the system \eqref{eq:RGV}, we have
	\[
	\frac{|\nabla\Phi|^2}{\Phi}
	\leqslant\frac{1}{\alpha}\left( \frac{3\tau}{2} - c  \right)
	\]
\end{corollary}

\begin{proof}
By Proposition \ref{Hormander-estimate}, at the maximum of $b = |\nabla^{1,0}\bm\phi|^2$ we have 
$$
0 \geqslant- \Delta_g b \geqslant\left( 2\alpha b + c + (\tau-\Phi)\left[  
	\alpha(\tau- \Phi) - \frac{3}{2}
	\right] 
	+
	\Phi
	\right) b
$$ 
and therefore, since $b \geqslant0$,
$$
2b \leqslant\alpha^{-1}( - c + (\tau-\Phi)(3/2 - 
	\alpha(\tau- \Phi)) - \Phi) \leqslant\frac{3\tau}{2} - c.
$$
The proof follows from Lemma \ref{lem:Phiphi}.
\end{proof}

\begin{remark}
Observe that existence of solutions of \eqref{eq:RGV} implies that $\tau > 2 N > 2$ (see Lemma \ref{lem:KR}, Theorem \ref{th:B-GP}, and Theorem \ref{th:AGGP}). Combined with the condition $\alpha > 0$, it implies that $\tau > \frac{2}{3}c$. 
\end{remark}

The combination of equation\eqref{Scalar-curvature-formula} and Corollary \ref{derivative-estimate} implies the main result of this section.

\begin{theorem}\label{bound-on-scalar}
For any solution $(\eta, g, \Phi)$ of the system \eqref{eq:RGV}, we have:
\begin{itemize}

\item (\emph{scalar curvature estimate})
\[
	c\leqslant S_g\leqslant\frac{(3+2\alpha\tau)\tau}{2}
	\]
	
\item (\emph{state function estimate})
\begin{equation}\label{Laplacian-Phi-equation}
-\frac{\tau^2}{4} \leqslant-\Delta_g\Phi = \frac{|\nabla\Phi|^2}{\Phi} - \Phi(\tau-\Phi) \leqslant	\frac{1}{\alpha}\left( \frac{3\tau}{2} - c  \right)
\end{equation}

\end{itemize}
\end{theorem}
 
To finish this section, we prove a higher order estimate for the scalar curvature 
of the conformal metric $k=e^{2\alpha \Phi}g$. For this, note that we have the formula 
\begin{equation}\label{scalarofauxmetric}
S_k
=
e^{-2\alpha\Phi}(S + \alpha \Delta_g \Phi)
=
e^{-2\alpha\Phi}(c+\alpha\tau(\tau-\Phi)).
\end{equation}

\begin{lemma}\label{lem:curvature-derivative-estimate}
For any solution $(\eta, g, \Phi)$ of the system \eqref{eq:RGV}, we have:

\begin{itemize}

\item $c e^{-2\alpha\tau} \leqslant S_k \leqslant c+\alpha\tau^2$,
	
\item $|\nabla_k S_k|_k^2  \leqslant
\frac{3}{2}\alpha\tau^2\left( 2c + 2\alpha\tau^2 + \tau \right)^2$.
\end{itemize}
\end{lemma}

\begin{proof}
The proof follows from Lemma \ref{lem:bound-for-Phi} and Corollary \ref{derivative-estimate}, combined with
\begin{equation}\label{eq:curvature-derivative-estimatebis}
\begin{split}
|\nabla_k S_k|_k^2 
& =
e^{-2\alpha\Phi}|d S_k|_g^2
=
\alpha^2\left( 2c + 2\alpha\tau(\tau-\Phi) + \tau \right)^2
e^{-6\alpha\Phi}
|\nabla_g \Phi|_g^2\\
& \leqslant
\frac{3}{2}\alpha\tau^2\left( 2c + 2\alpha\tau^2 + \tau \right)^2.
\end{split}
\end{equation}
\end{proof}



\vspace{0.2cm}

\subsection{Cheeger-Gromov convergence for gravitating vortices}\label{subsec:CG}

Let us take a sequence $(\omega_n, h_n)$ of solutions of the gravitating vortex equations with coupling constant 
$$
\alpha_n\to \alpha\in (0, \tfrac{1}{\tau N})
$$
and fixed volume $\Vol_{\omega_n} = 2\pi$. Let $(\eta_n, g_n, \Phi_n)$ be the corresponding solution to Riemannian gravitating vortex equations \eqref{eq:RGV}.
Our next goal is to construct a limiting solution $(g_\infty',\eta_\infty',\Phi_\infty')$ of \eqref{eq:RGV}, as we approach the boundary of our continuity path \eqref{eq:continuitypath}. In this section, we start by constructing a limiting metric $k'_\infty$ which, as we will see in Section \ref{subsec:close}, is related to $g_\infty'$ by conformal rescaling. Firstly, let us study the compactness of the family $g_n$ by using the \emph{a priori} estimates proved in Section \ref{subsec:apriori}. Let $k_n=e^{2\alpha_n \Phi_n} g_n$ be the auxilliary Riemannian metric, for which the curvature and covariant derivative of the curvature are uniformly bounded by Lemma \ref{lem:curvature-derivative-estimate}.  For the remainder of this section, convergence refers to subsequential convergence, unless specified otherwise.



\begin{lemma}\label{lem:CG}
By Cheeger-Gromov compactness, there exist a sequence of diffeomorphisms $\varphi_n$ on $S^2$ such that $\varphi_n ^{*} k_n \rightarrow \tilde k_\infty$ in $C^{2,\beta}$ as tensors as $n \to \infty$.
\end{lemma}

\begin{proof}
By Lemma \ref{lem:statebound}, the volume is bounded along the sequence, as
\[
2\pi \leqslant\Vol_{k_n} \leqslant2\pi e^{2\alpha_n\tau},
\]
and the diameter of $k_n$ is bounded from above by Bonnet's diameter estimate \cite[Chapter 6, Section 4, Lemma 21]{Petersen}:
\[
diam (S^2, k_n)\leqslant\frac{\pi}{\sqrt{c_n e^{-2\alpha_n \tau}}}=:
D_n.
\]
Furthermore, using the Relative Volume Comparison theorem \cite[Chapter 9, Section 1, Lemma 36]{Petersen}, by comparing $k_n$ with $\mathbb{R}^2$, the volume ratio  is bounded from below:
\[
\frac{\Vol_{k_n} B(p, r)}{\pi r^2}
\geqslant
\frac{\Vol_{k_n} B(p, D_n)}{\pi D_n^2}
\geqslant
\frac{2}{D_n^2}\;, \;\; \forall p\in S^2, r\in (0, D_n],
\]

 and thus Cheeger-Gromov-Taylor's Theorem \cite[Theorem 4.7]{CGT} gives an estimate on the lower bound of the injectivity radius:
\[
inj(S^2, k_n)\geqslant i_0
\]
for $i_0>0$ independent of $n$. By Theorem \ref{bound-on-scalar} and Lemma \ref{lem:curvature-derivative-estimate}, we have uniform bounds for $S_k$ and $|\nabla_{k_n} S_{k_n}|_{k_n}$ along the sequence, and therefore the statement follows by Cheeger-Gromov compactness \cite{Ham}.
\end{proof}

\begin{remark}
Notice that, in the previous proof, $i_0$ could go to $0$ if $c$ goes to $0$, so the estimate here collapses in the case $c\to 0$. However, `openness' at $\alpha=\frac{1}{\tau N}$ (see Lemma \ref{lem:openst} and Lemma \ref{lem:openpolyst}) means that we only need to restrict to the case $c\geqslant c'>0$. In the proof of Theorem \ref{thm{existence}} where we deal with $c=0$, this difficulty is overcome by a different geometric argument (see Section \ref{sec:eb}). 
\end{remark}

The key disadvantage of Lemma \ref{lem:CG} is that the $\varphi_n$ are only known to be diffeomorphisms of $S^2$, which might not respect the almost complex structure $J_0$ on $\mathbb{P}^1$ (in particular, the limit metric $\tilde k_\infty$ may not be compatible with $J_0$). 
Thus, even if we  prove that Cheeger-Gromov convergence provides a limiting solution of \eqref{eq:gravvortexeq1}, the points $p_j$ move along the sequence to $p_{j,n}=\varphi_n^{-1} (p_j)$, and we lose track of the divisor $D = \sum_j n_j p_j$ as $n \to \infty$. To remedy this shortcoming, we prove that the sequence of diffeomorphisms in Cheeger-Gromov convergence can be chosen to be holomorphic on $\mathbb{P}^1$. To achieve this goal, we need some preparatory material, which is probably well-known to experts, but does not seem to appear in the existing literature.

\begin{proposition}\label{prop:BersChern}
	Any $C^{k,\beta}$ ($k\in \mathbb{N}_+, \beta\in (0,1)$) almost complex structure on $S^2$ is the pull-back of the standard almost complex structure by a $C^{k+1,\beta}$ diffeomorphism. 
\end{proposition}
\begin{proof}
Let $J_0$ be the standard almost complex structure on $S^2$, i.e. $(S^2, J_0)=\mathbb{P}^1$ and let $\mathfrak{a}_0$ be the smooth structure underlying $J_0$. Let $g$ be a $C^{k,\beta}$ ($k\geqslant0, \beta\in (0,1)$) metric on $(S^2,\mathfrak{a}_0)$ defining the given $C^{k,\beta}$ almost complex structure $J$. By Chern-Bers' Theorem \cite{Bers,Chern}, there exists an atlas of isothermal coordinates for $J$ (obtained by solving the Beltrami equation) which is of class $C^{k+1,\beta}$ with respect to $\mathfrak{a}_0$. The Uniformisation Theorem implies there exists a biholomorphic map $f:(S^2, J) \rightarrow (S^2, J_0)$, i.e. $f^*J_0=J$. The map $f$ is $C^\infty$ with respect to the smooth structures of the left and the right manifolds, but only of class $C^{k+1,\beta}$ with respect to $\mathfrak{a}_0$. 
\end{proof}


Our next goal is to prove the following result. We are grateful to Yohsuke Imagi for conversations about the proof.

\begin{lemma}\label{lem:slice}
Let $J_i$ be a sequence of $C^{2,\beta}$ almost complex structures on $S^2$ converging in $C^{2,\beta}$ sense to another almost complex structure $J$ on $S^2$. Then, there exists a sequence of $C^{3,\beta}$ diffeomorphisms $f_i:S^2\to S^2$ such that
	\begin{enumerate}
		\item[1)] $f_i^*J_i=J$;
		\item[2)] the $f_i$ converge to $\Id$ in $C^{3,\beta}$ sense.
	\end{enumerate}
\end{lemma}

\begin{proof}
Let us denote by $\Sigma$ the Riemann surface $(S^2,J)$. We use the notation in the proof of Proposition \ref{prop:BersChern}. By hypothesis, $J$ is of class $C^{2,\beta}$ on $(S^2,\mathfrak{a}_0)$. We shall work in the smooth structure $\mathfrak{a}$ induced by $J$, where the almost complex structure is a smooth tensor. Consider the elliptic operator
\begin{equation}
\label{eq:complex}
\Omega^0(T^{1,0}\Sigma)  \lra{\dbar} \Omega^{0,1}(T^{1,0}\Sigma).
\end{equation}
Since $\Sigma$ has genus zero we have that $H^{0,1}(T^{1,0}\Sigma) = 0$, and therefore by Hodge theory
$$
\Omega^{0,1}(T^{1,0}\Sigma) = \operatorname{Im} \dbar \; \oplus H^{0,1}(T^{1,0}\Sigma) =  \operatorname{Im} \dbar.
$$
Thus, \eqref{eq:complex} is an isomorphism. Here we have used that $\operatorname{ker} \dbar^* = H^{0,1}(T^{1,0}\Sigma)$ and $\operatorname{ker} \dbar  =  \Omega^{0,1}(T^{1,0}\Sigma)$, which follows by dimensional reasons. 

Let $\mathcal{J}$ be the space of almost complex structures on $\Sigma$, whose tangent space at $J$ can be identified with $\Omega^{0,1}(T^{1,0}\Sigma)$. Consider the $C^1$ map between Banach spaces (for the $C^{2,\beta}$ completions)
\begin{equation}\label{eq:Frechetmap}
	\begin{split}
	\Omega^0(T^{1,0}\Sigma)  
	& \longrightarrow \mathcal{J}\\
	\bm v 
	& \mapsto f_{\bm v}^*J
	\end{split}
\end{equation}
where $f_{\bm v}$ denotes the flow of $\bm v$ (identified with a real vector field on $\Sigma$) at time $1$. Then, the Fr\'echet differential of \eqref{eq:Frechetmap} coincides with \eqref{eq:complex}, and therefore by the inverse function theorem it is invertible. Given now our sequence $J_i$ as in the statement, by taking isothermal coordinates for $J$ of class $C^{3,\beta}$ on $(S^2,\mathfrak{a}_0)$, we obtain that the $J_i$ are also of class $C^{2,\beta}$ on $\Sigma$, and converge to $J$ in $C^{2,\beta}$ sense. Then, there exists a sequence of $C^{2,\beta}$ vector fields $\bm v_i$ on $\Sigma$ converging to $0$, such that
$$
J_i = f_{\bm v_i}^*J.
$$ 
By standard ODE theory, the diffeomorphisms $f_{\bm v_i}$ are of class $C^{3,\beta}$ on $\Sigma$. Taking isothermal coordinates for $J$ of class $C^{3,\beta}$ on $(S^2,\mathfrak{a}_0)$, it follows that the $f_{\bm v_i}$ satisfy the conditions in the statement.
\end{proof}

We are now ready to prove the main result of this section, which provides an amended Cheeger-Gromov limit. As we will see, the family of diffeomorphism $\varphi_n$ in the Cheeger-Gromov  convergence in Lemma \ref{lem:CG} differs from a family of holomorphic automorphism $\sigma_n \in \SL(2,\CC)$ by some $C^{3,\beta}$ controlled diffeomorphisms. We use the same notation as in Lemma \ref{lem:CG}.

\begin{lemma}\label{lem:CGamm}
There exist a sequence $\sigma_n \in \SL(2,\CC)$ and a $C^{2,\beta}$ K\"ahler metric $k_\infty'$ on $\mathbb{P}^1$ with volume $2\pi$, such that $\sigma_n ^{*} k_n:=k_n' \rightarrow k_\infty'$ in $C^{2,\beta}$ as tensors as $n \to \infty$.
\end{lemma}

\begin{proof}
By assumption, the almost complex structure $J_{g_n}$ is fixed to be $J_0$ along the sequence. Let $\varphi_n$ be the sequence of diffeomorphisms in Lemma \ref{lem:CG}, which satisfy
\begin{equation}
\begin{split}
\tilde k_n= \varphi_n^* k_n 
& \rightarrow_{C^{2,\beta}} \tilde k_\infty\\
\varphi_n^* J_{k_n}=\varphi_n^* J_0 
&\rightarrow_{C^{2,\beta}} \tilde J_\infty
\end{split}
\end{equation}
By Proposition \ref{prop:BersChern} and Lemma \ref{lem:slice}, there exists $C^{3,\beta}$ diffeomorphisms $\psi_n$ and $f$ on $S^2$ such that 
\begin{equation}
\begin{split}
\psi_n^*\varphi_n^* J_0 
& = \tilde J_\infty= f^* J_0\\
\psi_n
& \rightarrow_{C^{3,\beta}} \Id
\end{split}
\end{equation}
We conclude that 
\begin{equation}
\tilde \sigma_n: = \varphi_n \circ \psi_n\circ f^{-1} \in \Aut(S^2, J_0) 
= \operatorname{PGL}(2,\mathbb{C}),
\end{equation}
and hence $\tilde \sigma_n^*k_n:=k_n'$ is a family of K\"ahler metrics on $\mathbb{P}^1$ converging in the $C^{2,\beta}$ sense to $k_\infty'=(f^{-1})^*\tilde k_\infty$. Finally, since the $\SL(2,\CC)$-action on $\PP^1$ factorizes through the natural covering map $\SL(2,\CC) \to \operatorname{PGL}(2,\mathbb{C})$, we can choose $\sigma_n \in \SL(2,\CC)$ covering $\tilde \sigma_n$ as in the statement. 
\end{proof}

\subsection{Estimate on Green's and state functions}\label{Green-function}

In order to construct our limiting solution $(g_\infty',\eta_\infty',\Phi_\infty')$ as $n \to \infty$, we need to control the function $\log \Phi_n'$, for $\Phi'_n = \sigma_n^* \Phi_n$ along the sequence. This will require a detailed analysis of the limit of a family of Green's functions for a $C^{2,\beta}$ convergent family of metrics on a Riemann surface. 

Let $(\Sigma, J)$ be a compact Riemann surface, let $\omega_0$ be a K\"ahler form on $\Sigma$ with volume $\nu$, and $dd^c=- dJd=2 i\partial\bar\partial$, then there exists a (uniquely determined) Green's function $G_{\omega_0}(\cdot,\cdot)$ satisfying 
\begin{equation}\label{Eq:Green-function-definition}
\begin{split}
	dd^c G_{\omega_0}(\cdot, Q) & = [Q] - \frac{1}{\nu} \omega_0,\\
\int_\Sigma G_{\omega_0}(\cdot, Q) \omega_0 & =0,
\end{split}
\end{equation}
for all $Q \in \Sigma$. We notice that locally the Green's function is asymptotic to  $\frac{1}{4\pi} \log |z|^2$ (the Green's function on $\mathbb{C}$) where $z$ is a local holomorphic coordinate of $\Sigma$ centered at $Q$. The dependence of $G_\omega$ on the K\"ahler form $\omega$ in a fixed cohomology class is as follows: if we take another K\"ahler metric $\omega=\omega_0+dd^c\lambda$ with the normalization $\int_\Sigma \lambda(\omega+\omega_0)=0$, then \cite[Proposition 1.3, Chapter II]{Lang} shows
\begin{equation}\label{different-Green-function}
G_{\omega}(P,Q)=G_{\omega_0}(P, Q) -\frac{1}{\nu} \big( \lambda(P)+\lambda(Q)  \big)
\end{equation}
As a consequence 
\[
\sup_{P,Q\in \Sigma} G_{\omega}
\leqslant
\sup_{P,Q\in \Sigma} G_{\omega_0} + \frac{2}{\nu}\sup_{P\in \Sigma} | \lambda | 
\]

\begin{proposition}\label{upper-bound-of-Green-function}
	Let $(\Sigma, J_0, \omega_0)$ be a Riemann surface with a $C^\beta$ K\"ahler metric $\omega_0$, for $\Lambda>1$ define 
	\[
	\mathcal{K}_{\Lambda}=\{ \omega=\omega_0+dd^c \lambda | \lambda\in C^{2,\beta},\; \Lambda^{-1}\omega_0\leqslant\omega\leqslant\Lambda \omega_0 \}
	\]
	\begin{enumerate}
		\item
		Then there exists a constant $K=K(\omega_0, \beta,\Lambda)>0$ such that $\forall \omega\in \mathcal{K}_\Lambda$, 
		\[
		\sup_{P, Q\in\Sigma} G_{\omega}\leqslant K
		\]	
		\item 
		If $\omega_i\in [\omega_0]$ converges to $\omega_0$ in $C^\beta$ sense, and $Q_i\to Q$, then $G_{\omega_i}(\cdot, Q_i)$ converges in $C^{2,\beta}$ sense to $G_{\omega_0}(\cdot, Q)$ on any compact subset away from $Q$. 
	\end{enumerate}
\end{proposition}

\begin{proof}
	
	(1). For $\omega=\omega_0+dd^c\lambda'\in \mathcal{K}_\Lambda$ with the normalization $\int_\Sigma \lambda' \omega_0=0$, a standard elliptic estimate shows 
	$
	||\lambda' ||_{C^0}\leqslant C(\omega_0)(\Lambda-1)
	$.
	Since $\lambda=\lambda'-\frac{1}{2V}\int_\Sigma \lambda' \omega$ satisfies the normalization condition 
	$\int_\Sigma \lambda(\omega_0 + \omega)=0$, we have 
	\[
	||\lambda||_{C^0}\leqslant\frac{3}{2}C(\Lambda, \omega_0).
	\]
	Therefore, the desired upper bound of $G_\omega$ is obtained since $G_{\omega_0}$ is bounded above. 
	
	(2). In the case $\omega_i=\omega_0+dd^c\lambda_i'$ satisfies $||\omega_i-\omega_0||_{C^\beta}\to 0$, using the normalization $\int_\Sigma \lambda_i' \omega_0=0$, 
	\[
	||\lambda_i'||_{C^{2,\beta}}\leqslant C(\omega_0) ||\omega_i-\omega_0||_{C^\beta}\to 0.
	\]
	Thus, the correctly normalized K\"ahler potential $\lambda_i$ (as above) satisfies also $||\lambda_i||_{C^{2,\beta}} \to 0$. We notice that 
	\begin{equation*}
	\begin{split}
	G_{\omega_i}(\cdot, Q_i) - G_{\omega_0}(\cdot, Q)
	& =
	\Big( G_{\omega_i}(\cdot, Q_i) - G_{\omega_0}(\cdot, Q_i) \Big)
	+ \Big( G_{\omega_0}(\cdot, Q_i) - G_{\omega_0}(\cdot, Q) \Big)\\
	& =	  
	-\frac{1}{\nu}\Big( \lambda_i(\cdot) + \lambda_i(Q_i) \Big)
	+ \Big( G_{\omega_0}(\cdot, Q_i) - G_{\omega_0}(\cdot, Q) \Big).     
	\end{split}
	\end{equation*}
	Therefore, the $C^{2,\beta}$ convergence away from $Q$ following from $||\lambda_i'||_{C^{2,\beta}}\to 0$ and the convergence property of $G_{\omega_0}$ under the convergence of its poles. 
\end{proof}
\vspace{0.3cm}

We now return to the situation of our interest. With the notation of Lemma \ref{lem:CGamm}, we want to control the function $\log \Phi_n'$, for $\Phi'_n = \sigma_n^* \Phi_n$, as $n \to \infty$. We will use the structural equation
\begin{equation}\label{log-Phi-equation}
\Delta_{k_n'} \log \Phi_n'
=
(\tau-\Phi_n') e^{-2\alpha_n \Phi_n'}- 4\pi\sum_j n_j \delta_{p_{j,n}'},
\end{equation}
which can be easily derived from the last equation in \eqref{eq:RGV}. For simplicity, denote $\nu_n' =\Vol_{k_n'}$ and $p_{j,n}'=\sigma_n^{-1}(p_j)$. By Lemma \ref{lem:bound-for-Phi}, $2\pi\leqslant\nu_n' \leqslant2\pi e^{2\alpha_n \tau}$. Consider the solution $G_n'$ of the following Green's function equation
\begin{equation}\label{multiple-Green-function}
\Delta_{k_n'} G_n'
=
\frac{N}{\nu_n'} -  \sum_j n_j\delta_{p_{j,n}'},
\end{equation}
given by by summing up the Green's functions with a simple pole $ G_{\omega_{k_n'}}(\cdot, p_{j,n}') : = G_{n,j}'$, i.e. 
\[
G_n'
=
\sum_{j} n_j  G_{n,j}',
\]
where $ G_{n,j}'$ satisfies 
\begin{equation}
dd^c  G_{n,j}' 
=
[p_{j,n}']- \frac{1}{\nu_n'} \omega_{k_n'}
\end{equation} 
with the normalization condition $\int_{\mathbb{P}^1} G_{n,j}' \omega_{k_n'}=0$ for each $j$ (see Equation \eqref{Eq:Green-function-definition}).  Taking the difference of the equations \eqref{log-Phi-equation} and  \eqref{multiple-Green-function} and denoting $ v_n'=\log \Phi_n' - 4\pi G_n'$, we obtain

\begin{equation}\label{error-Laplacian-equation}
\Delta_{k_n'}  v_n'
=
(\tau- \Phi_n') e^{-2\alpha_n \Phi_n'} - \frac{4\pi N}{\nu_n'}.
\end{equation}

The uniform $C^0$ bound of the `right hand side' given by Lemma \ref{lem:bound-for-Phi}, together with the $C^{2,\beta}$ bounded coefficient of $\Delta_{k_n'}$ actually enable us to get the estimate 
\begin{equation}\label{oscillation-bound}
C_1
\leqslant
v_n' 
- \frac{1}{\nu_n'}\int_{\mathbb{P}^1}  v_n'\vol_{k_n'}
\leqslant
C_2.
\end{equation}

This bound, in particular, implies that the oscillation of $v_n'= \log \Phi_n' - 4 \pi G_n'$ is bounded above by $C_2-C_1$. To get the estimate on  $ v_n'$ and $\log \Phi_n'$, we need some estimates on $G_n'$ (which resembles the singularities of $\log \Phi_n'$).

Now we pass to a subsequence (still use $n$ as its index) such that $p_{j,n}'=\sigma_n^{-1}(p_j) \rightarrow p_{j,\infty}'$ (where the $p_{j,\infty}'$'s need not be distinct), and we denote $D_\infty'= \sum_j n_j p_{j,\infty}'$. Let $G_\infty'$ be the Green's function for the divisor $D_\infty'$ under the limit $C^{2,\beta}$ metric $k_\infty'$ and $\nu_\infty'=\Vol_{k_\infty'}$, i.e 
\begin{equation}\label{limit-Green-function}
\Delta_{k_\infty'} G_\infty'
=
\frac{N}{\nu_\infty'} -  \sum_j n_j \delta_{p_{j,\infty}'}.
\end{equation}
then we have the following convergence:
\begin{proposition}[estimates on Green's functions]~\label{Green-function-estimates}
	\begin{enumerate}
		\item There exists $K'>0$ (independent of $n$) such that \footnote{Actually, Remark 3.3 of \cite{BM} gives an explicit upper bound: $G_n' \leqslant24 N \frac{diam^2(\mathbb{P}^1, k_n')}{\Vol(\mathbb{P}^1, k_n')}$.} 
		\[
		\sup_{\mathbb{P}^1} G_n'\leqslant K'.
		\]
		\item $G_n'$ converges to $G_\infty'$ in $C^{2,\beta}$ sense on any compact subset $\Omega\subset \mathbb{P}^1 \backslash D_\infty'$. 
	\end{enumerate}
\end{proposition}

\begin{proof}
	Applying Proposition \ref{upper-bound-of-Green-function} to the present situation where 
	$k_n'=\sigma_n^* k_n$ is a family of K\"ahler metric (by rescaling some constant bounded above and below (more precisely, consider $\frac{2\pi}{\nu_n'} k_n'$), which does not affect the estimate on Green's function) in the same K\"ahler class as $k_\infty' $. We know that $k_n'$ converges to $k_\infty'$ in $C^{2,\beta}$ and thus Proposition \ref{upper-bound-of-Green-function} shows
	\begin{equation}
		\sup_{\mathbb{P}^1} G_{\omega_{k_n'}} \leqslant K.
	\end{equation}
	Thus follows 
	\begin{equation}
		\sup_{P, Q} G_n' \leqslant K' : = N K.
	\end{equation}

	Moreover, according to Proposition \ref{upper-bound-of-Green-function}, on any compact subset away from $D_\infty'$ the function $G_n'$ converges to $G_\infty'$ in $C^{2,\beta}$ sense. 
\end{proof}

We conclude this section with the following uniform estimate on the state function.

\begin{proposition}\label{Prop:uniform-bound-C0}
	There exists constant $C_5>0$ such that 
	\[
	\left| \log \Phi_n' - 4 \pi G_n'\right|_{C^0(\mathbb{P}^1)} \leqslant C_5. 
	\]
\end{proposition}
\vspace{0.2cm}




\begin{proof}
By Lemma \ref{lem:bound-for-Phi}, there exists $x_0\in \mathbb{P}^1$ such that $\Phi_n'(x_0) = \tau-2N$. Therefore, it follows that for all $x\in \mathbb{P}^1$, 
\begin{equation}
\begin{split}
\log \Phi_n' (x) - 4\pi  G_n'(x)
& \geqslant
\log \Phi_n'(x_0) - 4\pi G_n'(x_0) 
- (C_2-C_1)\\
& \geqslant
\log (\tau-2N) - 4\pi K' - (C_2-C_1)
: =C_3,
\end{split}
\end{equation}
where we use the uniform upper bound of $G_n'$ in Proposition \ref{Green-function-estimates}.  This inequality implies the following lower bound of $\log \Phi_n'$ in terms of $G_n'$:
\begin{equation}\label{log-Phi-lower}
\log \Phi_n'  \geqslant4 \pi G_n' + C_3.
\end{equation}



Finally, the oscillation bound \eqref{oscillation-bound} together with Lemma \ref{lem:bound-for-Phi} imply
\begin{equation}\label{log-Phi-upper}
\log \Phi_n'
\leqslant
4 \pi G_n'
+
\log \tau
+ 
\frac{1}{\nu_n'} \int_{\mathbb{P}^1} - 4 \pi G_n' \vol_{k_n'} + C_2
=4\pi G_n' +\log \tau + C_2.
\end{equation}

\end{proof}

\subsection{Construction of the limit solution}\label{subsec:close}

We are now ready to prove that the sequence constructed in Lemma \ref{lem:CGamm} converges to a smooth solution of the Riemannian gravitating vortex equations. Note that the points $p_j$ move to $p_{j,n}'=\sigma_n^{-1} (p_j)$. Recall our notation $D_\infty'$ introduced right before Equation \eqref{limit-Green-function}. 

\begin{proposition}\label{prop:limit}
The sequence $(g_n',\eta_n',\Phi_n') = \sigma_n^*(g_n,\eta_n,\Phi_n)$ converges (up to taking a subsequence) as $n \to \infty$ in $C^{1,\beta}$ sense to a smooth solution  $(g_\infty',\eta_\infty',\Phi_\infty')$ of \eqref{eq:RGV} on $(S^2,D_\infty')$, with coupling constant $\alpha$ and symmetry breaking parameter $\tau$.
\end{proposition}

\begin{proof}
It follows from the third equation in \eqref{eq:RGV} that
\begin{equation}\label{conformal-factor-Laplacian}
\begin{split}
\Delta_{k_n'} \Phi_n'
& =
-\frac{|\nabla_{k_n'}\Phi_n' |_{k_n'}^2}{\Phi_n'}
+ \Phi_n' (\tau - \Phi_n') e^{-2\alpha_n \Phi_n'}\\
& =
\sigma_n^*\left( 
-\frac{|\nabla \Phi_n|_{g_n}^2}{\Phi_n} e^{-2\alpha_n\Phi_n}
+ \Phi_n (\tau - \Phi_n) e^{-2\alpha_n \Phi_n}
\right)
\end{split}
\end{equation}
away from the divisor $(\sigma_n)^{-1}(D)$, and since all the terms on the left and the right are smooth, this equation holds globally on $\mathbb{P}^1$.  

Since $k_n' = \sigma_n^* k_n$ converges to $k_\infty'$ in $C^{2,\beta}$ sense, and the `right hand side' of the above equation is uniformly bounded in $C^0$, we conclude that $\Phi_n'$ is uniformly bounded in $C^{1,\beta}$ by the standard $W^{2,p}$ estimate. Therefore, $g_n'=e^{-2\alpha_n\Phi_n'} k_n'$ converges to $ g_\infty'$ in $C^{1,\beta}$ sense, and $\eta_n' = \frac{1}{2}(\tau - \Phi_n') \vol_{g_n'}$ converges to $\eta_\infty'$ in $C^{1,\beta}$ sense.  In one sentence, 
\begin{equation}
\sigma_n^*(g_n, \Phi_n, \eta_n)
\rightarrow_{C^{1,\beta}}
(g_\infty', \Phi_\infty', \eta_\infty').
\end{equation}


Looking back into Equation \eqref{error-Laplacian-equation} and using $||v_n'||_{C^0(\mathbb{P}^1)}\leqslant C_5$ from Proposition \ref{Prop:uniform-bound-C0}, we obtain that the `right hand side' is uniformly bounded in the $C^\beta$ sense on $\PP^1$, and thus the Schauder estimate implies 
\begin{equation}
|| v_n' ||_{C^{2,\beta}}\leqslant C_6.
\end{equation}
Hence, $v_n'$ has a $C^{2,\beta}$ limit $v_\infty'$ as $n \to \infty$, and $v_\infty'$ satisfies (in the classical sense on $\mathbb{P}^1$)
\begin{equation}\label{limit-error-function}
\Delta_{k_\infty'}  v_\infty'
=
(\tau-\Phi_\infty') e^{-2\alpha \Phi_\infty'} - \frac{4\pi N}{\nu_\infty'}.
\end{equation}

  Using Proposition \ref{Green-function-estimates}, we know that $\Phi_n'=e^{v_n' + 4\pi G_n'}$ converges to $\Phi_\infty' =e^{ v_\infty' +  4 \pi G_\infty'}$ in $C^{2,\beta}$ on any compact subset $\Omega\subset\mathbb{P}^1\backslash D_\infty'$ (and the convergence is $C^{1,\beta}$ on $\mathbb{P}^1$).  Combining Equation \eqref{limit-error-function} with \eqref{limit-Green-function}, we conclude $\log\Phi_\infty'=v_\infty'+4\pi G_\infty'$ that the equation
\begin{equation}
\Delta_{k_\infty'} \log \Phi_\infty'
= (\tau- \Phi_\infty') e^{-2\alpha \Phi_\infty'} - 4\pi \sum_j n_j \delta_{p_{j,\infty}'}
\end{equation}
is satisfied in the classical sense on $\mathbb{P}^1\backslash D_\infty'$, and in the distributional sense on $\mathbb{P}^1$ (see \eqref{eq:RGVweak}) and likewise, for the un-rescaled metric $g_\infty'$, 

\begin{equation}
\Delta_{g_\infty'} \log \Phi_\infty'
= (\tau- \Phi_\infty') - 4\pi \sum_j n_j \delta_{p_{j,\infty}'}.
\end{equation}



The convergence of $\Phi_n'$ to $\Phi_\infty'$ implies $g_n'$ converges to $g_\infty'$ in $C^{2,\beta}$ sense away from $D_\infty'$ and $C^{1,\beta}$ sense on $\mathbb{P}^1$. The consequence is that $S_{g_\infty'} =\lim_{n\to \infty} S_{g_n'}$ on $\mathbb{P}^1\backslash D_\infty'$, and it follows simply from  Equation \eqref{eq:RGV} that on $\mathbb{P}^1\backslash D_\infty'$,

\begin{equation}\label{eq:limit-scalar-curvature-equation}
S_{g_\infty'} + \alpha (\Delta_{g_\infty'}+\tau)(\Phi_\infty' -\tau) = c.
\end{equation}

Moreover, the equation

\begin{equation}
\eta_\infty'+ \frac{1}{2}(\Phi_\infty'-\tau) \vol_{g_\infty'}=0
\end{equation}
is satisfied in classical sense on $\mathbb{P}^1$. 
The regularities of the data is: 
\begin{equation}
\begin{split}
\eta_\infty'  
& \in C^{1,\beta}(\mathbb{P}^1)\\
g_\infty' 
&\in C^{1,\beta}(\mathbb{P}^1)\cap C^{2,\beta}_{loc}(\mathbb{P}^1\backslash D_\infty')\\
\Phi_\infty'
&\in C^{1,\beta}(\mathbb{P}^1) \cap C^{2,\beta}_{loc}(\mathbb{P}^1\backslash D_\infty').
\end{split}
\end{equation}
Smoothness of the solution follows by a direct application of Lemma \ref{lem:regular2}.

\end{proof}


With the previous results at hand, we are ready for the proof of our main result.

\begin{proof}[Proof of Theorem \ref{th:main}]
By Proposition \ref{prop:limit}, we obtain a set of smooth limit data 
\[
(\eta_\infty', g_\infty', \Phi_\infty', D_\infty')
= 
\lim_{n\to \infty} \sigma_n^* (\eta_n, g_n, \Phi_n, D)
\] 
which is a solution to \eqref{eq:RGV}. By Lemma \ref{lem:KR}, we have a solution $(\omega_\infty', h_\infty')$ of \eqref{eq:gravvortexeq1} with divisor $D_\infty'$ on $\mathbb{P}^1$. Here $\omega_\infty'= g_\infty' (J_0\cdot, \cdot)$, and $h_\infty'$ is the Hermitian metric with curvature form $\eta_\infty'$. Theorem \ref{th:AGGP} implies then $D_\infty'$ is \emph{polystable} in the GIT sense. The divisor $D_\infty'$ is in the orbit closure of $D=\sum_j n_j p_j\in S^N(\mathbb{P}^1)$ under the $\SL(2,\mathbb{C})$-action on $S^N(\mathbb{P}^1)$. Thus, uniqueness of polystable orbits inside one orbit closure (following general GIT), implies $D_\infty' \in \SL(2,\mathbb{C})\cdot D$ since we assume $D$ is polystable. This verifies that
\begin{equation}
D_\infty'
=\sigma^*(D)
\end{equation}
for some $\sigma\in \SL(2,\mathbb{C})$. Finally, $(\sigma^{-1})^* \left( \omega_\infty', h_\infty' \right)$ gives the solution of the gravitating vortex equations with parameters $\alpha$, and $\tau$ and with the holomorphic section being the defining section of the divisor $D$. Therefore, the set 
\begin{equation*}
S = \{\alpha \in (0,\tfrac{1}{\tau N}]  \; \textrm{such that \eqref{eq:continuitypath} has a smooth solution }(\omega, h) \text{ with }\Vol_\omega = 2 \pi  \}.
\end{equation*}
is closed and by Lemma \ref{lem:openst} and Lemma \ref{lem:openpolyst} is also open, and therefore $S = (0,\tfrac{1}{\tau N}]$.
\end{proof}

To finish this section, we make some comments about the limit $\alpha\to 0$. Notice that the estimates in Lemma \ref{lem:curvature-derivative-estimate} for $S_k$ and $|\nabla_k S_k|_k$ are still valid for $\alpha \to 0^+$. Actually,
\begin{equation}
\begin{split}
S_k 
& 
\longrightarrow 2,\\
|\nabla_k S_k|_k 
& 
\longrightarrow 0.
\end{split}
\end{equation}
in $C^0$ sense as $\alpha\to 0^+$. 

For any $\alpha_n\to 0^+$, let $(\omega_{\alpha_n}, h_{\alpha_n})$ be a solution to the gravitating vortex equations on $(\mathbb{P}^1, L, \bm\phi)$. Then, the argument 
in Lemma \ref{lem:CGamm} applies. More precisely, there exists a sequence of automorphism $\sigma_n\in \SL(2,\CC)$ such that 
\[
\sigma_n^*\omega_{k_{\alpha_n}} :=\omega_{k_{\alpha_n}'}
\longrightarrow_{C^{2,\beta}} \omega_{FS}.
\]

Since $iF_{h_n'}=\frac{1}{2}(\tau-\Phi_n')\omega_{k_{\alpha_n}'}$ are uniformly bounded (with respect to $\omega_{FS}$), the metric $h_n'$ differs from $h_{\omega_{FS}}$ by some factor $e^{2f_n'}$ with $f_n'$ being $C^{1,\beta}$ bounded (for any $\beta\in (0,1)$).  The sections $\bm\phi_n':=\sigma_n^*\bm\phi\in H^0(L)$ are a family of holomorphic sections, which measured under the family of Hermitian metrics $h_n':=\sigma_n^*h_{\alpha_n}$ are uniformly bounded, i.e. 

\[
|\bm\phi_n'|^2_{h_n'} \leqslant\tau.
\]

Thus, the norms measured in a fixed Hermitian metric $|\bm\phi_n'|_{h_{\omega_{FS}}}$ are also uniformly bounded. Because of the holomorphicity, we can take a subsequence $\bm\phi_{n_j}'$ such that $\bm\phi_{n_j}'\to \bm\phi_\infty'$ in $C^\infty$ sense. 
By taking subsequence we obtain that $f_n'\rightarrow_{C^{1,\beta}} f_\infty'$, and therefore $h_\infty'$ is a $C^{1,\beta}$ weak solution to the vortex equation
\begin{equation}\label{eq:vortexFS}
iF_{h_\infty'} + \frac{1}{2}(|\bm\phi_\infty'|^2_{h_\infty'} -\tau) \omega_{FS}=0.
\end{equation}
Standard regularity for Abelian vortices implies that $h_\infty'$ is smooth.  By the integral condition 
\[
\int_{\mathbb{P}^1} (\tau - |\bm\phi_\infty'|_{h_\infty'}^2) \omega_{FS} = 4\pi N
\]
and the numerical condition $\tau>2N$, we conclude that $\bm\phi_\infty'$ is a nonzero section of $L$.

The gravitating vortex equations decouple at $\alpha=0$ and therefore in the limit $\alpha \to 0$ we are not able to conclude that $D_\infty' \in \SL(2,\mathbb{C}) \cdot D$. As stated before, the striking difference between gravitating vortices and Abelian vortices is that the existence of the latter does not impose any stability condition on the divisor.

\section{Proof of Theorem \ref{thm{existence}}}
\label{sec:eb}

In this section we give the proof of Theorem \ref{thm{existence}}. For this, a new continuity path is introduced, and combined with Theorem \ref{thm:HSY} (which provides our starting point), Lemma \ref{lem:openst}, and a refinement of the estimates in Section \ref{sec:close}.

\subsection{Proof of Theorem \ref{thm:HSY} in the stable case}
\label{subsec:HSY}

We start by explaining how Theorem \ref{thm:HSY} follows from the main results in \cite{HanSohn,Yang,Yang3}. Here we shall focus on the case that $D$ is stable, and postpone the strictly polystable case to Section \ref{sec:polyHSY}. In this setup, the condition $c = 0$ in \eqref{eq:gravvortexeq1} and \eqref{eq:constantc} is equivalent to $\alpha = \tfrac{1}{\tau N}$. 

We fix the Fubini-Study K\"ahler form $\omega_0=\omega_{FS}$ on $\mathbb{P}^1$ with volume $2\pi$ and take a Hermitian metric $h_0$ on $L$ such that $i F_{h_0} = N \omega_0$. Then, the Einstein-Bogomol'nyi equations for $\omega = \omega_0 + dd^c u =  (1- \Delta u) \omega_0, h=e^{2f}h_0$, with $u,f\in C^\infty(\PP^1)$ are equivalent to (see  \cite[Equation (2.6)]{AlGaGaPi})
\begin{equation}\label{eq:KWtype0}
\begin{split}
\Delta f + \frac{1}{2}(e^{2f}|\bm\phi|^2-\tau)e^{4\alpha \tau f - 2 \alpha e^{2f}|\bm\phi|^2} & = - N,\\
\Delta u + e^{4\alpha \tau f - 2 \alpha e^{2f}|\bm\phi|^2} & = 1.
\end{split}
\end{equation} 
Here, $\Delta$ is the Laplacian of $\omega_0$ and $|\phi|$ is the pointwise norm with respect to $h_0$ on $L$. Therefore, the Einstein-Bogomol'nyi equations reduce to a single PDE for the function $f$, given by the first equation in \eqref{eq:KWtype0}. Note that $\omega = (1- \Delta u) \omega_0$ implies $1 - \Delta u >0$, which is compatible with the last equation in \eqref{eq:KWtype0}.

In order to solve this PDE, Yang considers in \cite{Yang} the following rescaled equation for an undetermined parameter $\lambda>0$: 
\begin{equation*}
\label{eqn:HS-lambda}
\tag{$EB_\lambda$}
\begin{split}
\Delta f_{\lambda} 
& =
\frac{1}{2 \lambda}
(\tau - |\bm\phi|^2 e^{2f_{\lambda}} ) e^{4\alpha\tau f_{\lambda} -2\alpha |\bm\phi|^2 e^{2f_{\lambda}}} - N.
\end{split}
\end{equation*}
If $f_\lambda$ satisfies \eqref{eqn:HS-lambda}, then the following pair satisfies the Einstein-Bogomol'nyi equations
\begin{equation}\label{eq:lambdafamily}
(\omega_\lambda,h_\lambda) = (\lambda^{-1} e^{4\alpha\tau f_\lambda - 2\alpha|\bm\phi|_{h_\lambda}^2} \omega_0, h_0 e^{2f_\lambda}).
\end{equation}


\begin{theorem}[Yang's Existence Theorem]\label{th:Yang}
Assume that $\alpha = \tfrac{1}{\tau N}$. Then, there exists a solution of \eqref{eqn:HS-lambda} on $(\PP^1,L,\bm\phi)$ if one of the following conditions holds
\begin{enumerate}
\item $D = \frac{N}{2}p_1 + \frac{N}{2}p_2$, where $p_1 \neq p_2$ and $N$ is even, $\tau =1$, and $\lambda\in (0, \frac{1}{Ne})$. In this case the solution admits a $T^2$-symmetry.
\item $n_j < \frac{N}{2}$ for all $j$, and $\lambda>0$ is sufficiently small. 
\end{enumerate}
\end{theorem}

To prove part (2) of Theorem \ref{thm:HSY}, it remains to understand the asymptotic behavior of the total volume $\Vol_{\omega_\lambda}$ of the family of solutions \eqref{eq:lambdafamily} provided by Theorem \ref{th:Yang} when $\lambda \to 0$. This is a delicate question, which follows from a monotonicity relationship for Yang's solutions recently proved by Han-Sohn \cite{HanSohn}.

\begin{lemma}
	\label{lem:volume-behavior}
Let $(\omega_\lambda,h_\lambda)$ be the family of solutions of the Einstein-Bogomol'nyi equations provided by part (2) of Theorem \ref{th:Yang}. Then
$$
\lim_{\lambda\to 0} \text{Vol}_{\omega_\lambda} =  +\infty.
$$
\end{lemma}

\begin{proof}
Let $0<\lambda_1<\lambda_2$ sufficiently small, so that part (2) of Theorem \ref{th:Yang} applies. Then, by \cite[Lemma 2.4]{HanSohn} one has
$$
f_{\lambda_1} > f_{\lambda_2}.
$$
Thus, applying Lemma \ref{lem:statebound},
\begin{align*}
		\text{Vol}_{\omega_{\lambda_1}}
		& =
		\frac{1}{\lambda_1} \int_{S^2} e^{4\alpha\tau f_{\lambda_1} - 2\alpha|\bm\phi|_{h_{\lambda_1}}^2} \omega_0\\
		& > 
		\frac{1}{\lambda_1} \int_{S^2}
		e^{4\alpha\tau f_{\lambda_2}-2\alpha \tau}\omega_0
		\end{align*}
and the statement follows taking $\lambda_2$ fixed and $\lambda_1 \to 0$.
\end{proof}

The proof of part (2) of Theorem \ref{thm:HSY} follows combining the previous lemma with Theorem \ref{th:Yang} and Proposition \ref{prop:GIT}. 

\begin{remark}
The existence of solution for sufficiently small $\lambda$ in Theorem \ref{th:Yang} was recently extended by Han-Sohn \cite{HanSohn} to $\lambda\in (0, \lambda_c]$, for some abstractly determined threshold $\lambda_c$. In particular, they prove that the equation does not admit any solution for $\lambda>\lambda_c$ and admits multiple solutions for $\lambda\in (0,\lambda_c)$. The non-uniqueness of solution for \eqref{eqn:HS-lambda} with fixed $\lambda$ is a very interesting phenomenon. Due to the lack of simple geometric interpretation of $\lambda$ and the kind of branching behavior at the threshold, $\lambda$ might not be suitable parameter to understand the Einstein-Bogomolnyi equations. It is not clear what is the behavior of the volume of the second solution (constructed via Leray-Schauder degree theory) as $\lambda$ goes to $0$. The expected uniqueness of the solutions with fixed K\"ahler class modulo automorphisms \cite{AlGaGaPi} and Theorem \ref{thm{existence}} strongly suggest that the volume goes to $\frac{4\pi N}{\tau}$.
\end{remark}

\subsection{Strictly polystable case}\label{sec:polyHSY}

We explain next the proof of Theorem \ref{thm:HSY} in the case that $D$ is strictly polystable case, following \cite{Yang3}. Without loss of generality we can assume that $\tau = 1$ in \eqref{eq:KWtype0}, as if $(\omega, h)$ is a solution of the Einstein-Bogomol'nyi equations with $\tau = 1$ and coupling constant $\alpha=\frac{1}{N}$, then $(\tau^{-1}\omega, \tau h)$ is a solution with coupling constant $\widetilde\alpha = \tau^{-1}\alpha$ and symmetry breaking parameter $\tau$.

Given a strictly polystable divisor $D = \tfrac{N}{2} \cdot 0 + \tfrac{N}{2} \cdot \infty$, Yang \cite{Yang3} studied the existence of solutions of \eqref{eq:KWtype0} with $S^1$-symmetry by reducing the equation to the ODE initial value problem
\begin{equation}
\label{eqn:ODE}
\left\{
\begin{array}{cc}
u_{tt}
=
\frac{1}{\lambda}e^{2\alpha(u-e^u)}(e^u-1), 
& -\infty<t<+\infty\\
u(0) 
= -\mathfrak{b}, u_t(0)=0
\end{array}
\right.
\end{equation}
satisfying two asymptotic boundary conditions 
\[
\lim_{t\to + \infty} u_t(t)= - N, \qquad 
\lim_{t\to -\infty} u_t(t)  = N. 
\]
Using the shooting method of ODE, he showed that for each $\mathfrak{b}>0$ there is a unique parameter
\[
\lambda_{\mathfrak{b}} = \frac{1}{N e^{\mathfrak{b}+ e^{-\mathfrak{b}}}}
\] 
such that the above equation has a global solution $u^{\mathfrak{b}}$ with the asymptotic boundary conditions above. Taking
\begin{equation}
\begin{split}
g_{\mathfrak{b}}
& =
\frac{1}{\lambda_{\mathfrak{b}}} e^{2\alpha \left( u^{\mathfrak{b}}- e^{u^{\mathfrak{b}}}\right)} r^{-2} g_{euc},\\
\log |\bm\phi|_{h_{\mathfrak{b}}}^2
& = 
u^{\mathfrak{b}}, 
\end{split}
\end{equation}
the pair $(g_{\mathfrak{b}}, h_{\mathfrak{b}})$ is a $T^2$-symmetric solution to the Einstein-Bogomol'nyi equations with $\tau=1$. Here we identify $\mathbb{R}^2 \cong \PP^1 \backslash \{\infty\}$ via the stereographic projection. This proves part (1) of Theorem \ref{th:Yang}. The proof of part (1) of Theorem \ref{thm:HSY} follows now combining this last theorem with the following result. 

\begin{proposition}\label{prop:YangODEvolumeestimate}
Let $\omega_{\mathfrak{b}}$ be the K\"ahler form corresponding to $g_{\mathfrak{b}}$. Then, $\text{Vol}_{\omega_\mathfrak{b}}$ is a continuous function of $\mathfrak{b}>0$, and 
	\begin{enumerate}
		\item 
		\[
		\lim_{\mathfrak{b}\to 0^+} \text{Vol}_{\omega_{\mathfrak{b}}} 
		= 
		+\infty ;
		\]
		\item 
		\[
		\lim_{\mathfrak{b}\to +\infty} \text{Vol}_{\omega_{\mathfrak{b}}} 
		= 
		4\pi N. 
		\]
	\end{enumerate}
Consequently, for each $V\in (4\pi N, +\infty)$ there exists a $T^2$-symmetric pair $(g, h)$ solving the Einstein-Bogomol'nyi equations with symmetry-breaking parameter $\tau =1$ and $\text{Vol}_\omega = V$. 
\end{proposition}

\begin{proof}
	First, we have the formula 
	\begin{equation}
	\begin{split}
	\text{Vol}_{\omega_\mathfrak{b}}
	& =
	\frac{2\pi}{\lambda_\mathfrak{b}} 
	\int_0^{+\infty} e^{2\alpha \left( u^{\mathfrak{b}}- e^{u^{\mathfrak{b}}}\right)} r^{-2} \cdot r\mathrm{d}r \\
	& = 
	\frac{4\pi}{\lambda_\mathfrak{b}} 
	\int_0^{+\infty} e^{2\alpha \left( u^{\mathfrak{b}}(s)- e^{u^{\mathfrak{b}}(s)}\right)} \mathrm{d}s.
	\end{split}
	\end{equation}
	The proof of Lemma 3.1 in \cite{Yang3} about the continuity of $f$ yields directly that $\text{Vol}_{\omega_\mathfrak{b}}$ is continuous in $\mathfrak{b}\in (0, +\infty)$.

	The function $u^\mathfrak{b}$ is concave and even, and thus satisfies the estimate 
	\begin{equation}
	-\mathfrak{b}
	\geqslant 
	u^\mathfrak{b}(t)
	\geqslant -\mathfrak{b} - N |t|, \; \forall t\in \mathbb{R}. 
	\end{equation}
	The consequence is that for any fixed $T>0$, there holds $\forall\;  t\in [-T, T]$, $\mathfrak{b}\in (0,1]$:
	\begin{itemize}
		\item 
		\[
		0\geqslant u^\mathfrak{b}(t) \geqslant - 1 - NT;
		\]
		\item 
		\[
		N\geqslant u^\mathfrak{b}_t(t)\geqslant - N;
		\]
		\item 
		\[
		0 \geqslant u_{tt}^\mathfrak{b}(t)
		\geqslant 
		- N e^2;
		\] 
		\item 
		\[
		|u^\mathfrak{b}_{ttt}(t)|
		\leqslant N^2 e^2.
		\]
	\end{itemize}
	By the Arzela-Ascoli Theorem, as $\mathfrak{b}\to 0^+$, $u^\mathfrak{b}$ converges in $C^2$ sense to a function $\widehat u$ defined on $[-T, T]$ which satisfies equation \eqref{eqn:ODE} (on the restricted interval $[-T, T]$) with the initial values $\widehat u(0)=0\;, \widehat u_t(0) =0$. By the uniqueness of solutions, we conclude that $\widehat u\equiv 0$ on $[-T, T]$. As a consequence, 
	\begin{equation}
	\begin{split}
	\liminf_{\mathfrak{b}\to 0^+} 
	\text{Vol}_{\omega_\mathfrak{b}}
	& \geq
	\lim_{\mathfrak{b}\to 0^+} \text{Vol}_{\omega_\mathfrak{b}}\left( -T\leqslant t\leqslant T \right)\\
	& =
	\lim_{\mathfrak{b}\to 0^+}
	\frac{4\pi}{\lambda_\mathfrak{b}} 
	\int_0^{T} e^{2\alpha \left( u^{\mathfrak{b}}(s)- e^{u^{\mathfrak{b}}(s)}\right)} \mathrm{d}s\\
	& = 
	4\pi Ne^{1-2\alpha}T. 
	\end{split}
	\end{equation}
	Since $T$ could be chosen arbitrarily large, it follows that 
	\[
	\lim_{\mathfrak{b}\to 0^+} \text{Vol}_{\omega_\mathfrak{b}}=+\infty. 
	\]
	
	On the other hand, using the formula \eqref{Scalar-curvature-formula} for the scalar curvature 
	\[
	S_{g_\mathfrak{b}}
	=2\alpha |\mathrm{d}_A \bm\phi|_{h_{\mathfrak{b}}}^2 
	+ 
	\alpha (1-|\bm\phi|_{h_\mathfrak{b}}^2)^2
	\]
	and the Gauss-Bonnet formula, we have 
	
	\begin{equation}
	4\pi
	=
	\int_{S^2} S_{g_\mathfrak{b}} \mathrm{dvol}_{g_\mathfrak{b}}
	\geqslant 
	\alpha (1-e^{-\mathfrak{b}})^2 \text{Vol}_{\omega_\mathfrak{b}}. 
	\end{equation}
	Combined with the volume lower bound $\text{Vol}_{\omega_\mathfrak{b}} >4\pi N$ in Theorem \ref{th:AGGP} (notice that $\alpha N=1$), we get 
	\[
	\lim_{\mathfrak{b}\to +\infty}\text{Vol}_{\omega_\mathfrak{b}}
	=4\pi N.
	\]
\end{proof}


\subsection{The continuity method}

Let us introduce the continuity path which is used for the proof of Theorem \ref{thm{existence}}. In the sequel, we assume that the divisor $D$ is stable. Since it is easier to study the variations of a K\"ahler metric in a fixed K\"ahler class, in order to deform the total volume of a given solution we introduce the rescaled Einstein-Bogomol'nyi equations with parameter $\varepsilon > 0$ as the continuity parameter:
\begin{equation}\label{eq:Han-Sohn-generalization}
\begin{split}
iF_{\widetilde h} + \frac{1}{2\varepsilon} (|\bm\phi|_{\widetilde h}^2 - \tau)\widetilde\omega & = 0,\\
S_{\widetilde\omega} + \alpha ( \Delta_{\widetilde\omega} +  \frac{\tau}{\varepsilon} )( |\bm\phi|_{\widetilde h}^2 - \tau ) & = 0.
\end{split}
\end{equation}
Let $(\omega, h)$ be a solution of the Einstein-Bogomol'nyi equations constructed by Yang's Theorem \ref{th:Yang}. Then, $(\widetilde{\omega}, \widetilde{h})=\left( 2\pi \omega / \Vol_\omega, h\right)$ gives a solution to $\eqref{eq:Han-Sohn-generalization}$ with $\varepsilon =\frac{2\pi}{\text{Vol}_{\omega}}$ and total volume $2 \pi$. When there is no possibility of confusion, we will use the notation $(\widetilde{\omega}_{\varepsilon'}, \widetilde{h}_{\varepsilon'})$ for a solution of $\eqref{eq:Han-Sohn-generalization}$ with $\varepsilon = \varepsilon'$.

We would like to show that the system $\eqref{eq:Han-Sohn-generalization}$ has a solution $(\widetilde{\omega}_\varepsilon, \widetilde{h}_\varepsilon)$ with K\"ahler class $[\widetilde{\omega}_\varepsilon] = [\widetilde{\omega}]$ for any $\varepsilon \in \left(0, \frac{\tau}{2N}\right)$. Provided that this is true, $(\omega_\varepsilon,h_\varepsilon) = (\frac{1}{\varepsilon}\widetilde\omega_\varepsilon, \widetilde h_\varepsilon)$ gives a solution to the Einstein-Bogomol'nyi equations with $\text{Vol}_{\omega_\varepsilon} = \frac{2\pi}{\varepsilon}$, and hence the statement of Theorem \ref{thm{existence}} holds.

We define the set
\[
\mathcal{I}
=\left\{ \varepsilon \in \left(0, \frac{\tau}{2N}\right)| \exists (\widetilde\omega_\varepsilon, \widetilde h_\varepsilon)  \text{ solving } \eqref{eq:Han-Sohn-generalization} \text{ such that } [\widetilde\omega_{\varepsilon}] =  [\widetilde\omega]
\right\}.
\] 
Firstly, $\frac{2\pi}{\text{Vol}_{\omega}} \in \mathcal{I}$ by construction. Applying Lemma \ref{lem:openst} the existence of solutions of the Einstein-Bogomol'nyi equations is an open condition in the total volume, and hence by the construction above $\mathcal{I}$ is open. Now, let $\varepsilon_n$ be a sequence in $\mathcal{I}$ increasing or decreasing to $\widehat\varepsilon\in \left(0, \frac{\tau}{2N}\right)$. So, for each $n$ we have a solution $(\widetilde\omega_{\varepsilon_n}, \widetilde{h}_{\varepsilon_n})$ to $\eqref{eq:Han-Sohn-generalization}$ and thus a solution $(\omega_n, h_n):=(\omega_{\varepsilon_n}, h_{\varepsilon_n}) =\left(  \frac{1}{\varepsilon_n} \widetilde\omega_{\varepsilon_n}, \widetilde h_{\varepsilon_n}\right)$ to the gravitating vortex equations \eqref{eq:gravvortexeq1} with 
	\begin{equation}
	\label{ineq:volume-lower}
	\text{Vol}_{\omega_n}= \frac{2\pi}{\varepsilon_n}. 
	\end{equation}
To use the estimates obtained in Section \ref{sec:close}, we consider the convergence of $(\omega_n, h_n)$. In the next two sections we explain the crucial differences arising in the current situation, where $c=0$, compared with Section \ref{sec:close} where we deal with $c>0$. The lower bound on the scalar curvature in Proposition \ref{lower-bound-on-scalar} may be arbitrarily close to $0$, thus the diameter upper bound of $\omega_n$ does not follow directly. This was used in order to conclude the Cheeger-Gromov convergence of $\omega_n$. Another crucial difference is regarding Lemma \ref{lem:statebound}, where we assume that $\omega$ has volume $2\pi$. Instead, the volume is now varying along the sequence. This estimate is crucially used when deriving the uniform lower bound of the state function $\Phi_n$ (see the proof of Proposition \ref{Prop:uniform-bound-C0}). 

\subsection{Diameter upper bound}
Let $\Phi_n=|\bm\phi|_{ h_n}^2$ and define $k_n = e^{2\alpha \Phi_n} g_n$ as in Section \ref{sec:close}. Then, Lemma \ref{lem:curvature-derivative-estimate} implies that the $k_n$ admits a uniform bound on its curvature (between $0$ and some constant $K>0$ independent of $n$) and its covariant derivative. 

Intuitively, a sequence of metrics with bounded curvature cannot collapse everywhere unless the manifold is an almost flat manifold in Gromov's sense. In our situation, the manifold $S^2$ is definitely not a two dimensional almost flat manifold, thus the sequence $k_n$ must be uniformly non-collapsed at some point. More precisely, by the non-existence of an F-structure on $S^2$, due to the fact that $\chi(S^2)=2$ (see the first three lines of \cite[p. 310]{CG1}), and Cheeger-Gromov's Decomposition Theorem  \cite[Theorem 0.1]{CG2}, there exists a point $x_n\in S^2$ and a constant $\epsilon_0>0$ independent of $n$ such that 
		\begin{equation}
		inj(k_n, x_n)\geqslant \epsilon_0. 
		\end{equation}
		By Cheeger-Gromov's Compactness Theorem, there exists a  complete manifold $(X, k_\infty)$ and a point $p\in X$ such that 
		\[
		(S^2, k_n, x_n)\longrightarrow (X, k_\infty, p)
		\]
		in the $C^{2,\beta}$ pointed Cheeger-Gromov sense (possibly by passing to a subsequence). It is clear that the curvature of $k_\infty$ is non-negative and $\text{Vol}_{k_\infty}(X)$ is finite since the volumes $\text{Vol}_{k_n}$ are uniformly bounded from above by $\max\{ e^{2\alpha\tau}\cdot\Vol_\omega, e^{2\alpha\tau}
	\cdot \frac{2\pi}{\widehat\varepsilon}\}$. The well-known Calabi and Yau's \emph{linear volume growth estimate} for manifolds with non-negative Ricci curvature \cite{Yau} implies $X$ is actually compact and the above pointed Cheeger-Gromov convergence can be strengthened to Cheeger-Gromov convergence. Therefore, $X=S^2$ and we are now in the same setting as in Section \ref{subsec:CG}. Thus, we conclude that there exists a family $\sigma_n\in SL(2,\mathbb{C})$ and a subsequence of $k_{l_n}$ (still denoted by $k_n$) such that
		\begin{equation}
		\sigma_n^* k_n := k_n' \longrightarrow k_\infty' \text{ in }C^{2,\beta} \text{ sense as } n\to +\infty
		\end{equation}
		for a $C^{2,\beta}$ K\"ahler metric $k'_\infty$ on $\mathbb{P}^1$.  
		
		\subsection{State function lower bound}

	One of the key estimates
	\[
	\sup_{\mathbb{P}^1} \log \Phi \geqslant \log (\tau - 2N)
	\]
	 used in the proof of Proposition \ref{Prop:uniform-bound-C0} is derived from 
	\begin{equation}
	\label{eqn:key-estimate}
	\int_{\mathbb{P}^1} \Phi \omega = \tau \cdot \text{Vol}_\omega - 4\pi N
	\end{equation}
	in which $\text{Vol}_\omega$ is assumed to be constant $2\pi$.  In the current situation, the inequality	\eqref{ineq:volume-lower}  implies
	\[
	\sup_{\mathbb{P}^1} \log \Phi_n 
	\geqslant 
	\log \left( \tau - \frac{4\pi N}{\text{Vol}_{\omega_n}} \right)
	\geqslant 
	\log \left( \tau - 2N \widehat\varepsilon \right). 
	\]
	
Besides the above two differences, all other estimates in Section \ref{sec:close} hold and we can argue exactly as in the Proof of Theorem \ref{th:main}. Thus, we conclude that there exists a sequence $\gamma_n\in SL(2,\mathbb{C})$ such that for a subsequence $\varepsilon_{i_n}$ (still denoted by $\varepsilon_{n}$)
	\[
	(\eta_\infty, g_\infty, \Phi_\infty, D)
	= 
	\lim_{n\to \infty} \gamma_n^* (\eta_{\varepsilon_n}, g_{\varepsilon_n}, \Phi_{\varepsilon_n}, D)
	\]  
and the convergence is in $C^{1,\beta}$ sense. 	This in turn means $\gamma_n^*(\omega_{\varepsilon_n}, h_{\varepsilon_n}) \longrightarrow_{C^{1,\beta}} (\omega_\infty, h_\infty)$. The limit is actually smooth by Lemma \ref{lem:regular2} and gives rise to a solution to the Einstein-Bogomol'nyi equations with $\text{Vol}_{\omega_\infty}=\frac{2\pi}{\widehat\varepsilon}$. Then $(\widetilde\omega_{\widehat\varepsilon}, \widetilde h_{\widehat\varepsilon})=\left( \widehat\varepsilon \omega_\infty, h_\infty \right)$ solves $\eqref{eq:Han-Sohn-generalization}$ with $[\widetilde\omega_{\widehat\varepsilon}]=  [\widetilde\omega]$, and thus $\widehat\varepsilon\in \mathcal{I}$.  This finishes the proof. 

\end{document}